\documentclass[a4paper, 12pt, reqno]{amsart}
\usepackage{amsmath}
\usepackage{amssymb}
\usepackage{graphics}
\usepackage{mathrsfs}
\usepackage[pdflatex]{graphicx}
\usepackage{hyperref}
\usepackage[marginratio=1:1,tmargin=117pt, height=650pt]{geometry}
\usepackage{color}
\usepackage{amsthm}

\usepackage[all]{xy}

\usepackage{eurosym}

\usepackage[english]{babel}
\usepackage[T1]{fontenc}
\usepackage[latin1]{inputenc}
\usepackage{marginnote}

\usepackage{amsmath,amsfonts,amsthm,amssymb,verbatim,enumerate,quotes,graphicx}
\usepackage[flushmargin]{footmisc}
\usepackage[noadjust]{cite}
\usepackage{color}
\newcommand{\comments}[1]{}
\numberwithin{equation}{section}
\makeatletter
\def\blfootnote{\xdef\@thefnmark{}\@footnotetext}
\makeatother

\newcommand{\red}[1]{{\color{red} #1}}

\definecolor{orange}{rgb}{1,0.5,0}

\newcommand{\ds}{\displaystyle}

\theoremstyle{plain}
\newtheorem{thm}{Theorem}[section]

\newtheorem{lem}[thm]{Lemma}

\theoremstyle{definition}
\newtheorem{dfn}[thm]{Definition}
\newtheorem{ex}[thm]{Example}

\theoremstyle{remark} 

\newtheorem{rmk}[thm]{Remark}

\newcommand{\C}{{\mathbb{C}}}
\newcommand{\CR}{{\hat{\mathbb{C}}}}
\newcommand{\CS}{{\mathbb{C}^*}}
\newcommand{\B}{\mathcal B}
\newcommand{\R}{{\mathbb{R}}}
\newcommand{\Z}{{\mathbb{Z}}}
\newcommand{\N}{{\mathbb{N}}}

\renewcommand{\Re}{\operatorname{Re}}
\renewcommand{\Im}{\operatorname{Im}}

\makeatletter
\newcommand{\displaybump}{\hbox to \@totalleftmargin{\hfil}}
\makeatother

\begin{document}

\bibliographystyle{amsalpha}

\title[Escaping Fatou components of transcendental self-maps of $\C^*$]{Escaping Fatou components of transcendental self-maps of the punctured plane}
\author[D. Mart\'i-Pete]{David Mart\'i-Pete}
\address{Department of Mathematics\\ Faculty of Science\\ Kyoto University\\ Kyoto 606-8502\\ Japan}
\email{martipete@math.kyoto-u.ac.jp}
\date{\today}
\thanks{This research was supported by The Open University and by the Grant-in-Aid for Scientific Research JP16F16807 from the Japanese Society for the Promotion of Science.}

\maketitle

\begin{abstract}
We study the iteration of transcendental self-maps of $\C^*:=\C\setminus \{0\}$, that is, holomorphic functions $f:\C^*\to\C^*$ for which both zero and infinity are essential singularities. We use approximation theory to construct functions in this class with escaping Fatou components, both wandering domains and Baker domains, that accumulate to $\{0,\infty\}$ in any possible way under iteration. We also give the first explicit examples of transcendental self-maps of $\C^*$ with Baker domains and with wandering domains. In doing so, we developed a sufficient condition for a function to have a simply connected escaping wandering domain. Finally, we remark that our results also provide new examples of entire functions with escaping Fatou components. 
\end{abstract}

\section{Introduction}

Complex dynamics concerns the iteration of a holomorphic function on a Riemann surface $S$. Given a point $z\in S$, we consider the sequence given by its iterates $f^n(z)=(f\circ\ds\mathop{\cdots}^{n}\circ f)(z)$ and study the possible behaviours as $n$ tends to infinity. We partition $S$ into the \textit{Fatou set}, or stable set,
$$
F(f):=\bigl\{z\in S\ :\ (f^n)_{n\in\mathbb N} \mbox{ is a normal family in some neighbourhood of } z\bigr\}
$$
and the \textit{Julia set} $J(f):=S\setminus  F(f)$, where the chaotic behaviour takes place. We refer to a connected component of $F(f)$ as a \textit{Fatou component} of $f$. If $S\subseteq \CR$, $f:S\rightarrow S$ is holomorphic and $\CR\setminus S$ consists of essential singularities, then conjugating by a M\"obius transformation, we can reduce to one of the following three cases:
\begin{itemize}
\item $S=\CR:=\C\cup\{\infty\}$ and $f$ is a rational map;
\item $S=\C$ and $f$ is a transcendental entire function;
\item $S=\CS:=\C\setminus\{0\}$ and \textit{both} zero and infinity are essential singularities.
\end{itemize}
We study this third class of maps, which we call \textit{transcendental self-maps of} $\CS$. Such maps are all of the form
\begin{equation}
f(z)=z^n\exp\bigl(g(z)+h(1/z)\bigr),
\label{eqn:bhat}
\end{equation}
where $n\in\Z$ and $g,h$ are non-constant entire functions. We define the \textit{index}~of~$f$, denoted by $\textup{ind}(f)$, as the index (or winding number) of $f(\gamma)$ with respect to the origin for any positively oriented simple closed curve $\gamma$ around the origin; note that $\textup{ind}(f)=n$ in \eqref{eqn:bhat}. Transcendental self-maps of $\C^*$ arise in a natural way in many situations, for example, when you complexify circle maps, like the so-called Arnol'd standard family: $f_{\alpha,\beta}(z)=ze^{i\alpha}e^{\beta(z-1/z)/2}$, $0\leqslant \alpha\leqslant 2\pi,\ \beta\geqslant 0$ \cite{fagella99}. Note that if $f$ has three or more omitted points, then, by Picard's theorem, $f$ is constant and, consequently, a non-constant holomorphic function $f:\C^*\rightarrow \C^*$ has no omitted values. The book \cite{milnor06} is a basic reference on the iteration of holomorphic functions in one complex variable. See \cite{bergweiler93} for a survey on transcendental entire and meromorphic functions. Although the iteration of transcendental (entire) functions dates back to the times of Fatou \cite{fatou26}, R{\aa}dstr\"om \cite{radstrom53} was the first to consider the iteration of transcendental self-maps of $\C^*$. An extensive list of references on this topic can be found in \cite{martipete}.

We recall the definition of the \textit{escaping set} of an entire function $f$,
$$
I(f):=\{z\in\C\ :\ f^n(z)\rightarrow \infty \mbox{ as } n\rightarrow \infty\},
$$
whose investigation has provided important insight into the Julia set of entire functions. For polynomials, the escaping set consists of the basin of attraction of infinity and its boundary equals the Julia set. For transcendental entire functions, Eremenko showed that $I(f)\cap J(f)\neq \emptyset$, $J(f)=\partial I(f)$ and the components of $\overline{I(f)}$ are all unbounded \cite{eremenko89}. If $f$ is a transcendental self-map of $\C^*$, then the escaping set of $f$ is given by 
$$
I(f):=\{z\in\CS\ :\ \omega(z,f)\subseteq \{0,\infty\}\},
$$
where $\omega(z,f)$ is the classical omega-limit set $\omega(z,f):=\bigcap_{n\in\N}\overline{\{f^k(z)\ :\ k\geqslant n\}}$ with the closure being taken in $\CR$. In \cite{martipete1}, we studied the basic properties of $I(f)$ for transcendental self-maps of $\C^*$ and introduced the following notion. We define the \textit{essential itinerary} of a point $z\in I(f)$ as the symbol sequence \mbox{$e=(e_n)\in\{0,\infty\}^{\N_0}$} given by
$$
e_n:=\left\{
\begin{array}{ll}
0, & \mbox{ if } |f^n(z)|\leqslant 1,\vspace{5pt}\\
\infty, & \mbox{ if } |f^n(z)|>1,
\end{array}
\right.
$$
for all $n\in \N_0$. Then, for each sequence $e\in\{0,\infty\}^{\N_0}$, we consider the set of points whose essential itinerary is eventually a shift of $e$, that is,
$$
I_e(f):=\{z\in I(f)\ :\ \exists \ell,k\in\N_0,\ \forall n\in\mathbb{N}_0,\ |f^{n+\ell}(z)|>1\Leftrightarrow e_{n+k}=\infty\}.
$$
Observe that if $e_1,e_2\in\{0,\infty\}^{\N_0}$ satisfy $\sigma^m(e_1)=\sigma^n(e_2)$ for some $m,n\in\N_0$, where $\sigma$ is the Bernoulli shift  map (we say that $e_1$ and $e_2$ are \textit{equivalent}), then $I_{e_1}(f)=I_{e_2}(f)$ and, otherwise, the sets $I_{e_1}(f)$ and $I_{e_2}(f)$ are disjoint. Hence, the concept of essential itinerary provides a partition of $I(f)$ into uncountably many non-empty sets of the form $I_e(f)$ for some $e\in\{0,\infty\}^{\mathbb{N}_0}$. In \cite{martipete1}, we also showed that, for each $e\in\{0,\infty\}^{\N_0}$, we have $I_e(f)\cap J(f)\neq \emptyset$, $J(f)=\partial I_e(f)$~and the components of $\overline{I_e(f)}$ are all unbounded in $\C^*$, that is, their closure in $\CR$ contains zero or infinity. We say that $U$ is an \textit{escaping Fatou component} of $f$ if $U$ is a component of $F(f)\cap I(f)$.

As usual, the set of singularities of the inverse function, $\mbox{sing}(f^{-1})$, which consists of the critical values and the finite asymptotic values of $f$, plays an important role in the dynamics of $f$. In \cite{fagella-martipete} we studied the class 
$$
\B^*:=\{f \mbox{ transcendental self-map of } \CS\ :\ \mbox{sing}(f^{-1}) \mbox{ is bounded away from } 0,\infty\},
$$
which is the analogue of the Eremenko-Lyubich class~$\mathcal{B}$ considered in \cite{eremenko-lyubich92}. We proved that if $f\in \B^*$, then $I(f)\subseteq J(f)$ or, in other words, functions in the class~$\mathcal{B}^*$ have no escaping Fatou components. 

In this paper, we are concerned with transcendental self-maps of $\C^*$ that have escaping Fatou components. By normality, if $U$ is a Fatou component of a transcendental self-map $f$ of $\C^*$ and $U\cap I(f)\neq\emptyset$, then $U\subseteq I(f)$. Moreover, note that every pair of points in an escaping Fatou component $U$ have, eventually, the same essential itinerary, and hence we can associate an essential itinerary to $U$, which is unique up to equivalence. We mentioned before that $I_e(f)\cap J(f)\neq \emptyset$ for each sequence $e\in\{0,\infty\}^{\N_0}$ (see \cite[Theorem~1.1]{martipete1}). Therefore it is a natural question whether for each $e\in\{0,\infty\}^{\N_0}$, there exists a transcendental self-map of~$\C^*$ with a Fatou component in $I_e(f)$.


For both transcendental entire functions and transcendental self-maps of~$\C^*$, escaping Fatou components can be classified in the following two kinds: let $U$ be a Fatou component of $f$ and denote by $U_n$, $n\in\mathbb{N}$, the Fatou component of $f$ that contains $f^n(U)$, then we say that\vspace{-1pt}
\begin{itemize}
\item $U$ is a \textit{wandering domain} if $U_m\cap U_n=\emptyset$ for all $m,n\in\N$ such that $m\neq n$,\vspace*{1pt}
\item $U$ is a \textit{Baker domain} (or a preimage of it) if $U\subseteq I(f)$ and $U$ is (\textit{pre})\textit{periodic}, that is, $U_{p+m}=U_m$ for some $p\in \mathbb{N}$, the \textit{period} of $U$, and $m=0$ ($m>0$).\vspace{-1pt}
\end{itemize}
Note that not all wandering domains are in $I(f)$. For instance, Bishop~\cite[Theorem~17.1]{bishop15} constructed an entire function in the class $\B$ with a wandering domain whose orbit is unbounded but it does not escape. 

The first example of a transcendental entire function with a wandering domain was given by Baker \cite{baker63, baker76} and was an infinite product that had a sequence of multiply connected Fatou components escaping to infinity; see \cite{bergweiler-rippon-stallard13} for a detailed study of the properties of such functions. For holomorphic self-maps of $\C^*$, Baker \cite{baker87} showed that all Fatou components, except possibly one, are simply connected, and hence this kind of wandering domains cannot occur. Further examples of simply connected wandering domains of entire functions are due, for example, to Herman \cite[Example~2]{baker84} or Baker \cite[Example~5.3]{baker84}.  

Baker \cite{baker87} also constructed the first holomorphic self-map of $\C^*$ (which is entire) with a wandering domain that escapes to infinity. The first examples of trans\-cendental self-maps of $\C^*$ with a wandering domain are due to Kotus \cite{kotus90}, where the wandering domain accumulates to zero, infinity or both of them. In the same paper, Kotus also constructed an example with an infinite limit set (by adapting the techniques from \cite{eremenko-lyubich92}). Mukhamedshin \cite{mukhamedshin91} used quasiconformal surgery to create a trans\-cendental self-map of~$\C^*$ with a Herman ring and two wandering domains, one escaping to zero and the other one to infinity. Finally, Baker and Dom\'inguez \cite[Theorem~6]{baker-dominguez98} gave an example of a doubly connected wandering domain that is relatively compact in~$\C^*$ and all of whose images are simply connected and escape to infinity. 

In our notation, all the previous examples of wandering domains of transcendental self-maps of $\C^*$ had essential itinerary $e\in\{\overline{\infty},\overline{0}, \overline{\infty 0}\}$, where $\overline{e_1e_2\hdots e_p}$ denotes the $p$-periodic sequence that repeats $e_1e_2\hdots e_p$. The following result provides examples of transcendental self-maps of $\C^*$ that have a wandering domain with any prescribed essential itinerary $e\in\{0,\infty\}^{\N_0}$. 

\begin{thm}
\label{thm:wandering-domains}
For each sequence $e\in\{0,\infty\}^{\N_0}$ and $n\in\Z$, there exists a trans\-cendental self-map $f$ of $\C^*$ such that $\textup{ind}(f)=n$ and the set $I_e(f)$ contains a wandering domain.
\end{thm}

Observe that, in particular, in Theorem~\ref{thm:wandering-domains} we obtain functions with wandering domains whose essential itinerary is not necessarily a periodic sequence. 

The other type of escaping Fatou component is a Baker domain. The first example of a transcendental entire function with a Baker domain was already given by Fatou \cite{fatou26}: $f(z)=z+1+e^{-z}$. See \cite{rippon08} for a survey on Baker domains. 

A result of Cowen \cite{cowen81} on holomorphic self-maps of $\mathbb D$ whose Denjoy-Woff point lies on $\partial \mathbb D$ led to the following classification of Baker domains by Fagella and Henriksen \cite{fagella-henriksen06}, where $U/f$ is the Riemann surface obtained by identifying points of $U$ that belong to the same orbit under $f$: 
\begin{itemize}
\item a Baker domain $U$ is \textit{hyperbolic} if $U/f$ is conformally equivalent to $\{z\in\C\ :$\linebreak $-s<\Im z<s\}/\Z$ for some $s>0$;
\item a Baker domain $U$ is \textit{simply parabolic} if $U/f$ is conformally equivalent to $\{z\in\C\ :\ \Im z>0\}/\Z$;
\item a Baker domain $U$ is \textit{doubly parabolic} if $U/f$ is conformally equivalent to $\C/\Z$.
\end{itemize}
Note that this classification does not require $f$ to be entire and is valid also for Baker domains of transcendental self-maps of $\C^*$. K\"onig \cite{koenig99} provided a geometric characterisation for each of these types (see Lemma~\ref{lem:bd-koenig}). It is known that if $U$ is a doubly parabolic Baker domain, then $f_{|U}$ is not univalent, but if $U$ is a hyperbolic or simply parabolic Baker domain, then $f_{|U}$ can be either univalent or multivalent. Several examples of each type had been constructed and recently Bergweiler and Zheng completed the table of examples by constructing a transcendental entire function with a simply parabolic Baker domain in which the function is not univalent \cite[Theorem~1.1]{bergweiler-zheng12}.

The only previous examples of Baker domains of transcendental self-maps of~$\C^*$ that the author is aware of are due to Kotus \cite{kotus90}. She used approximation \mbox{theory} to construct two functions with invariant hyperbolic Baker domains whose points escape to zero and to infinity respectively. The following theorem provides functions with Baker domains that have any periodic essential itinerary $e\in\{0,\infty\}^{\N_0}$ and, in particular, Baker domains whose points accumulate to both zero and infinity.

\begin{thm}
\label{thm:baker-domains}
For each periodic sequence $e\in\{0,\infty\}^{\N_0}$ and $n\in\Z$, there exists a trans\-cendental self-map $f$ of $\C^*$ such that $\textup{ind}(f)=n$ and $I_e(f)$ contains a hyperbolic Baker domain. 
\end{thm}

\begin{rmk}
We observe that our method can be modified to produce doubly parabolic Bakers domains as well. However, the cons\-truction of simply parabolic Baker domains using approximation theory seems more difficult.
\end{rmk}


We also give the first explicit examples of transcendental self-maps of $\C^*$ with wandering domains and Baker domains. They all have the property that in a neighbourhood of infinity they behave like known examples of transcendental entire functions with wandering domains and Baker domains.

\begin{ex}
\label{ex:main-ex}
The following transcendental self-maps of $\C^*$ have escaping Fatou components:
\begin{enumerate}
\item[(i)] The function $f(z)=z\exp\left(\frac{\sin z}{z}+\frac{2\pi}{z}\right)$ has a bounded wandering domain that escapes to infinity (see Example~\ref{ex:wand-domain}).
\item[(ii)] The function $f(z)=2z\exp\bigl(\exp(-z)+1/z\bigr)$ has an invariant hyperbolic Baker domain that contains a right half-plane and whose points escape to infinity (see Example~\ref{ex:hyp-baker-domain}).
\item[(iii)] The function $f(z)=z\exp\left((e^{-z}+1)/z\right)$ has an invariant doubly parabolic Baker domain that contains a right half-plane and whose point escape to infinity (see Example~\ref{ex:dpar-baker-domain}).
\end{enumerate}
\end{ex} 

It seems hard to find explicit examples of functions with Baker domains and wandering domains with any given essential itinerary, but it would be interesting to have a concrete example of a function with an escaping Fatou component that accumulates to both zero and infinity. It also seems difficult to find explicit examples of functions with simply parabolic Baker domains.

We remark that in order to show that the function from Example~\ref{ex:main-ex}~(i) has a simply connected escaping wandering domain we introduced a new criterion (see Lemma~\ref{lem:Julia-in-annulus}) which is of more general interest. 

Let $f$ be a transcendental self-map of $\C^*$, then there exists a transcendental entire function $\tilde{f}$ such that $\exp \circ \,\tilde{f}=f\circ \exp$; we call $\tilde{f}$ a \textit{lift} of $f$. If the function $f$ has a wandering domain, then $\tilde{f}$ has a wandering domain, while if $f$ has a Baker domain, then $\tilde{f}$ has either a Baker domain (of the same type) or a wandering domain; see Lemmas~\ref{lem:semiconj-wd}~and~\ref{lem:semiconj-bd}. 

It is important that in both Theorems~\ref{thm:wandering-domains}~and~\ref{thm:baker-domains} we can choose the index of the function since, for example, if $\textup{ind}(f)\neq 1$, then $f$ does not have Herman rings. In \cite{martipete4} the author compares the escaping set of $f$ with that of a lift $\tilde{f}$ of $f$ according to $\textup{ind}(f)$.

Finally, observe that our constructions using approximation theory can also produce holomorphic self-maps of $\C^*$ of the form $f(z)=z^n\exp(g(z))$, with $n\in\Z$ and $g$ a non-constant entire function. In particular, they can provide new examples of transcendental entire functions with no zeros in $\C^*$ that have wandering domains and Baker domains.






\vspace{5pt}

\noindent 
\textbf{Structure of the paper.} In Sections 2 and 3 we prove that the functions from Example~\ref{ex:main-ex} have the properties that we state. In Section 4 we introduce the tools from approximation theory that we will use in Sections 5 and 6 to construct functions with escaping wandering domains and Baker domains respectively. Theorem~\ref{thm:wandering-domains} is proved in Section 5 and Theorem~\ref{thm:baker-domains} is proved in Section 6.

\vspace{5pt}

\noindent
\textbf{Notation.} In this paper $\N_0=\N\cup\{0\}=\{0,1,2,\hdots\}$ and, for $z_0\in \C$ and $r>0$, we define
$$
D(z_0,r):=\{z\in\C\ :\ |z-z_0|<r\},\quad \mathbb H_r:=\{z\in\C\ :\ \Re z>r\}. 
$$

\vspace{10pt}

\noindent
\textbf{Acknowledgments.} The author would like to thank his supervisors Phil Rippon and Gwyneth Stallard for their support and guidance in the preparation of this paper.

\section{Explicit functions with wandering domains}

\label{sec:explicit-wd}

As mentioned in the introduction, the author is not aware of any previous explicit examples of transcendental self-maps of $\C^*$ with wandering domains or Baker domains as all such functions were constructed using approximation theory.


Kotus \cite{kotus90} showed that transcendental self-maps of $\C^*$ can have escaping wandering domains by constructing examples of such functions using approximation theory. Here we give an explicit example of such a function by modifying a transcendental entire function that has a wandering domain. 

\begin{ex}
The function $f(z)=z\exp\bigl(\frac{\sin z}{z}+\frac{2\pi}{z}\bigr)$ is a transcendental self-map of~$\C^*$ which has a bounded simply connected wandering domain that escapes to infi\-nity (see Figure~\ref{fig:wand-domain}).
\label{ex:wand-domain}
\end{ex}

\begin{figure}[h!]
\includegraphics[width=.49\linewidth]{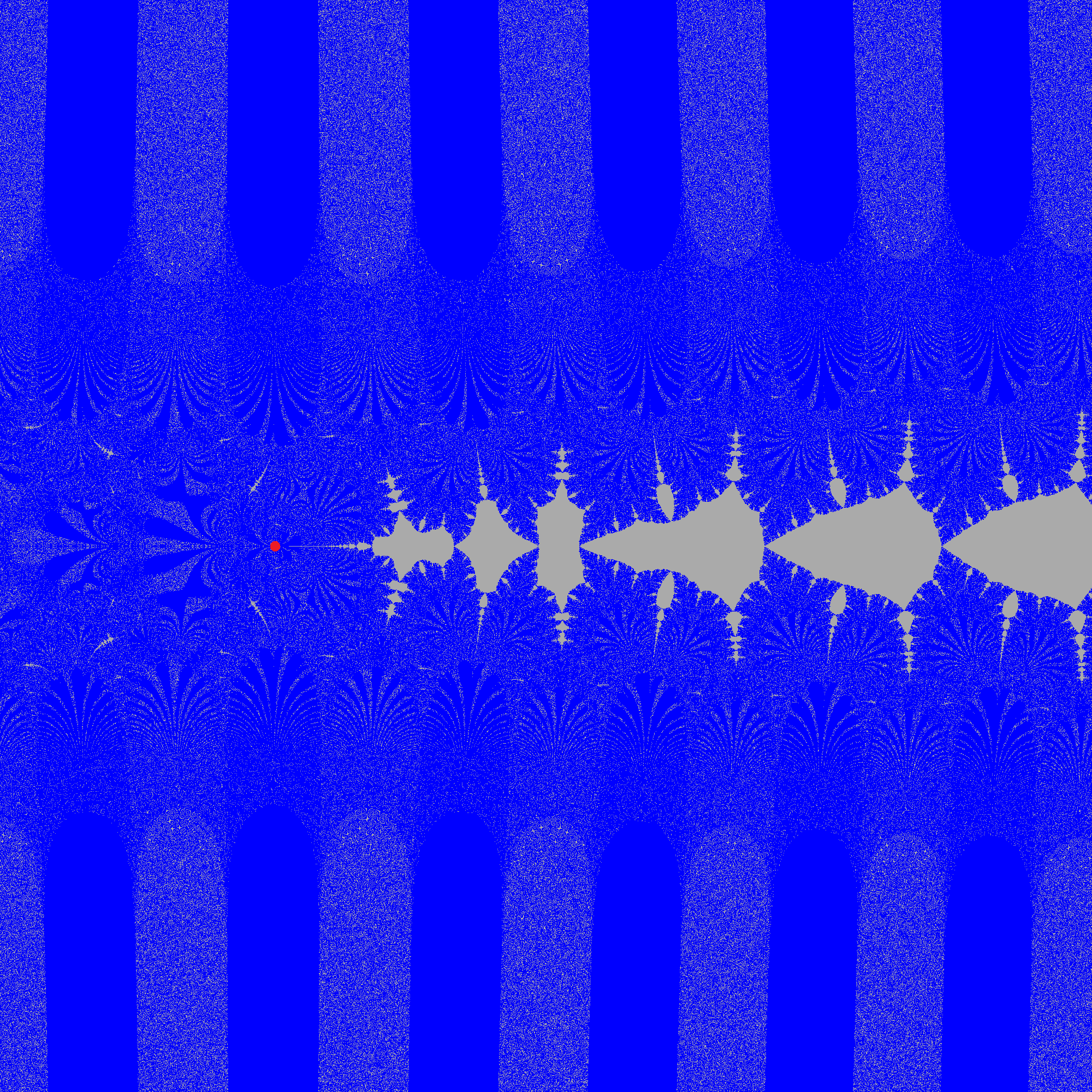}
\includegraphics[width=.49\linewidth]{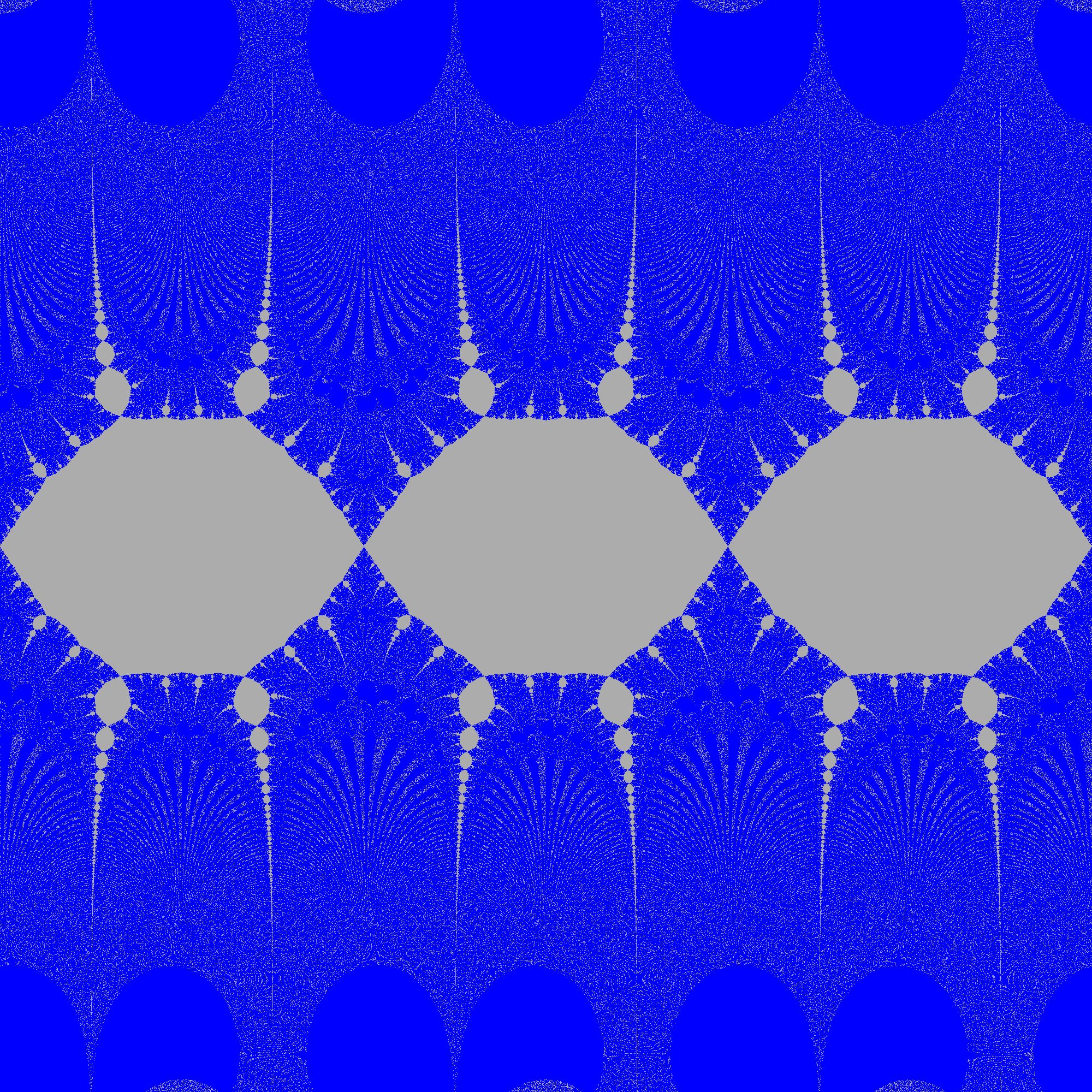}
\caption[Phase space of a transcendental self-map of $\C^*$ which has a wandering domain]{Phase space of the function $f(z)=z\exp\left(\frac{\sin z}{z}+\frac{2\pi}{z}\right)$ from Example~\ref{ex:wand-domain} which has a wandering domain. On the right, the wandering domain for large values of $\textup{Re}\, z$.}
\label{fig:wand-domain}
\end{figure}

Baker \cite[Example~5.3]{baker84} (see also \cite[Example~2]{rippon-stallard08}) studied the dynamics of the trans\-cendental entire function $f_1(z)=z+\sin z+2\pi$ that has a wandering domain containing the point $z=\pi$ that escapes to infinity. Observe that the function $f$ from Example~\ref{ex:wand-domain} satisfies that
\begin{equation}
f(z)=z+\sin z+2\pi+o(1)\quad \mbox{ as } \mbox{Re}\,z\rightarrow +\infty
\label{eq:ex-wand-domain}
\end{equation}
in a horizontal band defined by $|\textup{Im}\, z|<K$ for some $K>0$.

We first prove a general result which gives a sufficient condition that implies that a function has a bounded wandering domain (see Figure~\ref{fig:annuli-lemma}) using some of the ideas from \cite[Lemma~7(c)]{rippon-stallard08}. Given a doubly connected open set $A\subseteq \C$, we define the \textit{inner boundary}, $\partial_\textup{in}A$, and the \textit{outer boundary}, $\partial_\textup{out}A$, of $A$ to be the boundary of the bounded and unbounded complementary components of $A$ respectively.

\begin{lem}
Let $f$ be a function that is holomorphic on $\C^*$, let $M$ be an affine map, let $A$ be a doubly connected closed set in $\C^*$ with bounded complementary component $B$, and let $C\subseteq B$ be compact. Put
$$
A_n:=M^n(A),\quad B_n:=M^n(B)\quad \mbox{ and }\quad C_n:=M^n(C)\quad \mbox{ for } n\in\N_0,
$$
and suppose that 
\begin{itemize}
\item $A_n\cup B_n\subseteq \C^*$ for $n\in\N_0$,
\item the sets $\{B_n\}_{n\in\N_0}$ are pairwise disjoint,
\item $f(\partial_\textup{in}\,A_n)\subseteq C_{n+1}$ for $n\in\N_0$,
\item $f(\partial_\textup{out}\,A_n)\subseteq \C^*\setminus (A_{n+1}\cup B_{n+1})$ for $n\in\N_0$.
\end{itemize}
Then $f$ has bounded simply connected wandering domains $\{U_n\}_{n\in\N_0}$ such that
$$
\partial_\textup{in}\,A_n\subseteq U_n \quad \mbox{ and } \quad \partial U_n \subseteq A_n \quad \mbox{ for } n\in\N_0.
$$
\label{lem:Julia-in-annulus}
\end{lem}

\begin{figure}[h!]
\centering
\vspace*{-10pt}
\def\svgwidth{\linewidth}
\input{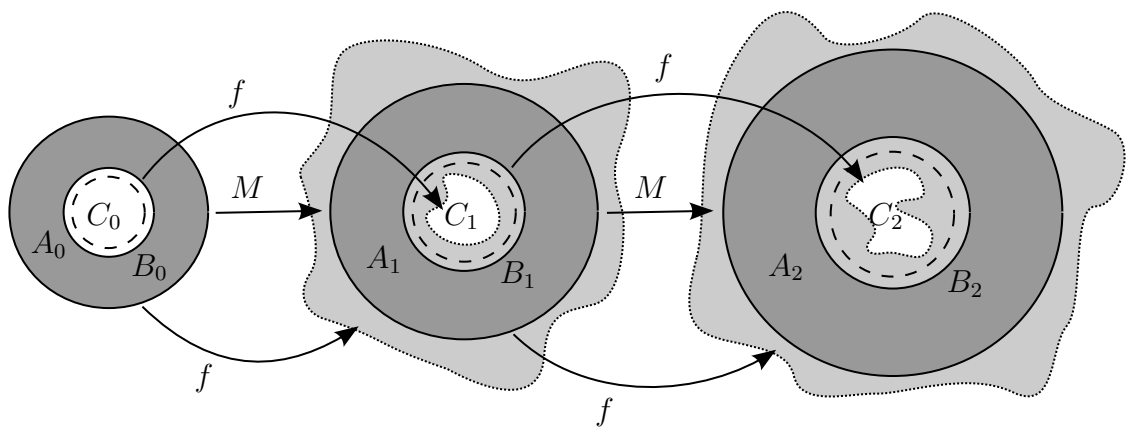}
\vspace{15pt}
\caption[Sketch of a construction which implies that a function has a wandering domain]{Sketch of the construction in Lemma \ref{lem:Julia-in-annulus}.}
\label{fig:annuli-lemma} 
\end{figure}

In order to prove this lemma, we first need the following result on limit functions of holomorphic iterated function systems by Keen and Lakic \cite[Theorem~1]{keen-lakic03}. 

\begin{lem}
Let $X$ be a subdomain of the unit disc $\mathbb D$. Then all limit functions of any sequence of functions $(F_n)$ of the form
$$
F_n:=f_n\circ f_{n-1}\circ \cdots \circ f_2\circ f_1\quad \mbox{ for } n\in\N,
$$
where $f_n:\mathbb D\to X$ is a holomorphic function for all $n\in \N$, are constant functions in $\overline{X}$ if and only if $X\neq\mathbb D$.
\label{lem:keen-lakic}
\end{lem}

We now proceed to prove Lemma~\ref{lem:Julia-in-annulus}.

\begin{proof}[Proof of Lemma~\ref{lem:Julia-in-annulus}]
Since $f(\overline{B_n})\subseteq C_{n+1}\subseteq B_{n+1}$, the iterates of $f$ on each set $\overline{B_n}$ omit more than three points and hence, by Montel's theorem, the sets $\{\overline{B_n}\}_{n\in\N_0}$ are all contained in $F(f)$. For $n\in \N_0$, let $U_n$ denote the Fatou component of $f$ that contains $\overline{B_n}$. We now show that the functions
$$
\Phi_{k}(z):=M^{-k}(f^k(z))\quad \mbox{ for } k\in\N_0,
$$
form a normal family in $U_n$ for all $n\in\N_0$.

Suppose first that the Fatou components $\{U_n\}_{n\in\N_0}$ are not distinct. Then there are two sets $B_m$ and $B_{m+p}$ with $m\in\N_0$ and $p>0$ which lie in the same Fatou components $U_m=U_{m+p}$. Then, since $f^p(B_m)\subseteq B_{m+p}$ and $B_n\to \infty$ as $n\to \infty$, $U_m$ must be periodic and in $I(f)$, and hence a Baker domain.

Let $z_m\in B_m$ and let $K$ be any compact connected subset of $U_m$ such that $K\supseteq B_m$. Then by Baker's distortion lemma (see \cite[Lemma~6.2]{martipete1} or \cite[Lemma~2.22]{martipete} for a proof of the version of this result that we use here), there exist constants \mbox{$C(K)>1$} and $n_0\in\N_0$ such that 
$$
|f^k(z)|\leqslant C(K) |f^k(z_m)|\quad \mbox{ for } z\in K,\ k\geqslant n_0.
$$
Since $M$, and hence $M^{-k}$, is an affine transformation, $M^{-k}$ preserves the ratios of distances, so
$$
|\Phi_k(z)|=|M^{-k}(f^k(z))|\leqslant C(K)|M^{-k}(f^k(z_m))|=C(K)|z_m'|
$$
where $z_m'\in B_m$ satisfies $M^k(z_m')=f^k(z_m)$. Hence the family $\{\Phi_k\}_{k\in \N_0}$ is locally uniformly bounded on $U_m$, and hence is normal on $U_m$.

Suppose next that the Fatou components $\{U_n\}_{n\in\N_0}$ are disjoint. In this case we consider the sequence of functions 
$$
\varphi_k(z):=M^{-(k+1)}(f(M^k(z)))\quad \mbox{ for } k\in\N_0,
$$
which are defined on $U_n$, for $n\in\N_0$. Then
\begin{equation}
\Phi_k(z) = (\varphi_{k-1} \circ \cdots \circ  \varphi_1 \circ \varphi_0)(z) = M^{-k}(f^k(z))\quad \mbox{ for } k\in\N_0.
\label{eq:annuli-lem-1}
\end{equation}
Since the Fatou components $\{U_n\}_{n\in\N_0}$ are pairwise disjoint and
$$
f^k(U_n)\subseteq U_{n+k},
$$
we deduce that
$$
f^k(U_n)\cap B_{n+k+1}=\emptyset
$$
and hence
$$
\Phi_k(U_n)\cap B_{n+1}=\emptyset\quad \mbox{ for } k,n\in\N_0.
$$
Thus $\{\Phi_k\}_{k\in\N_0}$ is normal on each $U_n$, by Montel's theorem, as required.

Now take $n\in\N_0$, and let $\{\Phi_{k_j}\}_{j\in\N_0}$ be a locally uniformly convergent subsequence of $\{\Phi_k\}_{k\in\N_0}$ on $B_n$. Note that
$$
M^k(B_n)=B_{n+k} \quad \mbox{ so } \quad f(M^k(B_n))\subseteq C_{n+k+1}
$$
and hence, for $k\in \N_0$, 
$$
\varphi_k(B_n)=M^{-(k+1)}(f(M^k(B_n)))\subseteq M^{-(k+1)}(C_{n+k+1})=C_n.
$$
We can now apply Lemma~\ref{lem:keen-lakic}, after a Riemann mapping from $B_n$ to the open unit disc $\mathbb D$, to deduce from \eqref{eq:annuli-lem-1} that there exists $\alpha_n\in \overline{B_n}$ such that, for all $z\in U_n$,
$$
\Phi_{k_j}(z)\to\alpha_n\quad \mbox{ as } j\to\infty.
$$

To complete the proof that $U_n$ is bounded by $\partial_\textup{out}\,A_n$ for all $n\in\N$, suppose to the contrary that there is a point $z_0\in\partial_\textup{out}\,A_n$ that lies in $U_n$ for some $n\in\N$. Let $\gamma\subseteq U_n$ be a curve that joins $z_0$ to a point $z_1\in B_n$. Since $\gamma$ is compact, $\Phi_{k_j}(\gamma)\to \alpha$ as $j\to\infty$ which contradicts the fact that $f^k(\gamma)\cap \partial_\textup{out}\,A_{n+k}\neq\emptyset$ for all $k\in\N$ (this follows from the hypothesis that $f(\partial_\textup{out}\,A_n)\subseteq (A_{n+1}\cup B_{n+1})^c$ for $n\in\N_0$). Thus, $\partial  U_n\subseteq A_n$ for all $n\in\N$, and so the proof is complete.
\end{proof}

We now use Lemma~\ref{lem:Julia-in-annulus} to show that the function $f$ from Example~\ref{ex:wand-domain} has a bounded wandering domain that escapes to infinity along the positive real axis.

\begin{proof}[Proof of Example~\ref{ex:wand-domain}]
 


The transcendental entire function $g(z)=z+\sin z$ has\linebreak superattracting fixed points at the odd multiples of $\pi$. For $n\in\N_0$, take\linebreak \mbox{$B_n := D((2n+1)\pi,r)$} and $C_n:=D((2n+1)\pi,r/2)$ for some $r>0$ sufficiently small that \mbox{$g(B_n)\subseteq C_{n}$} and put
$$
R_n:= \{z\in \C\ :\ |\textup{Re}\, z-(2n+1)\pi|\leqslant 3\pi/2,\ |\textup{Im}\, z|\leqslant 3\}.
$$
It follows from a straightforward computation that $g(\partial R_n)\subseteq R_n^c$ for all $n\in\N_0$ (see Figure~\ref{fig:spiral}). 


\begin{figure}[h!]
\centering
\def\svgwidth{.65\linewidth}
\input{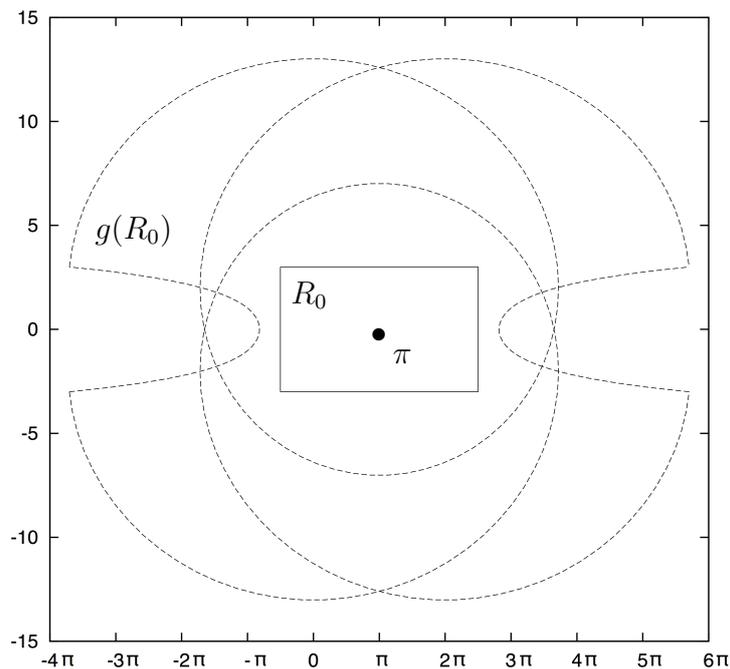} 
\caption[Image of a rectangle by the map~\mbox{$g(z)=z+\sin z$}]{Rectangle $R_0$ and its image under $g(z)=z+\sin z$.}
\label{fig:spiral}
\end{figure}

Then, by \eqref{eq:ex-wand-domain}, there exists $N\in\N_0$ such that $f(B_n)\subseteq C_{n+1}$ and $f(\partial R_n)\subseteq R_{n+1}^c$ for all $n>N$. Thus, we can apply Lemma~\ref{lem:Julia-in-annulus} to $f$ with $M(z)=z+2\pi$ and $A_n:=R_n\setminus B_n$ for $n>N$ and conclude that the function $f$ has wandering domains $U_n$ that contain $B_n$ and whose boundary is contained in $R_n$.
\end{proof}

The next lemma relates the wandering domains of a transcendental self-map of~$\C^*$ and a lift of it.

\begin{lem}
Let $f$ be a transcendental self-map of $\C^*$ and let $\tilde{f}$ be a lift of $f$. Then, if $U$ is a wandering domain of $f$, every component of $\exp^{-1}(U)$ is a wandering domain of $\tilde{f}$ which must be simply connected.
\label{lem:semiconj-wd}
\end{lem}
\begin{proof}
By a result of Bergweiler \cite{bergweiler95}, every component of $\exp^{-1}(U)$ is a Fatou component of $\tilde{f}$. Let $U_0$ be a component of $\exp^{-1}(U)$ and suppose to the contrary that there exist $m,n\in\N_0$, $m\neq n$, and a point $z_0\in\tilde{f}^m(U_0)\cap \tilde{f}^n(U_0)$. Then, there exists points $z_1,z_2\in U_0$ such that \vspace*{-5pt}
$$
f^m(e^{z_1})=\exp \tilde{f}^m(z_1)=\exp z_0=\exp \tilde{f}^n(z_2)=f^n(e^{z_2}).\vspace*{-5pt}
$$
Since $e^{z_1},e^{z_2}\in U$, this contradicts the assumption that $U$ is a wandering domain of $f$. Hence $U_0$ is a wandering domain of $\tilde{f}$. 

Finally, by \cite[Theorem~1]{baker87}, the Fatou component $U$ is either simply connected or doubly connected and surrounds the origin. Since the exponential function is periodic, taking a suitable branch of the logarithm one can show that the components of $\exp^{-1}(U)$ are simply connected.
\end{proof}

\begin{rmk}
Observe that the converse of Lemma~\ref{lem:semiconj-wd} does not hold. If $f$ is a transcendental self-map of $\C^*$ with an attracting fixed point $z_0$ and $A$ is the immediate basin of attraction of $z_0$, then there is a lift $\tilde{f}$ of $f$ such that a component of $\exp^{-1}(A)$ is a wandering domain.
\end{rmk}

If a transcendental self-map of $\C^*$ has an escaping wandering domain, then we can use the previous lemma to obtain automatically an example of a transcendental entire function with an escaping wandering domain.

\begin{ex}
The transcendental entire function $\tilde{f}(z)=z+\frac{\sin e^z}{e^z}+\frac{2\pi}{e^z}$, which is a lift of the function $f$ from Example~\ref{ex:wand-domain}, has infinitely many grand orbits of bounded wandering domains that escape to infinity.
\end{ex}



\section{Explicit functions with Baker domains}

\label{sec:explicit-bd}


We now turn our attention to Baker domains. As we mentioned in the introduction, Baker domains can be classified into hyperbolic, simply parabolic and doubly parabolic according to the Riemann surface $U/f$ obtained by identifying the points of the Baker domain $U$ that belong to the same orbit under iteration by the function $f$. K\"onig \cite{koenig99} introduced the following notation.

\begin{dfn}[Conformal conjugacy]
Let $U\subseteq \C$ be a domain and let \mbox{$f:U\to U$} be analytic. Then a domain $V\subseteq U$ is \textit{absorbing} (or \textit{fundamental}) for $f$ if $V$ is simply connected, $f(V)\subseteq V$ and for each compact set $K\subseteq U$, there exists $N=N_K$ such that $f^N(K)\subseteq V$.
Let $\mathbb H:=\{z\in\C\ :\ \textup{Re}\, z>0\}$. The triple $(V,\phi,T)$ is called a \textit{conformal conjugacy} (or \textit{eventual conjugacy}) of $f$ in $U$ if
\begin{enumerate}
\item[(a)] $V$ is absorbing for $f$;
\item[(b)] $\phi:U\to \Omega\in\{\mathbb H, \C\}$ is analytic and univalent in $V$;
\item[(c)] $T:\Omega\to\Omega$ is a bijection and $\phi(V)$ is absorbing for $T$;
\item[(d)] $\phi(f(z))=T(\phi(z))$ for $z\in U$.
\end{enumerate}
In this situation we write $f\sim T$.
\end{dfn}

Observe that properties (b) and (d) imply that $f$ is univalent in $V$. K\"onig also provided the following geometrical characterization of the three types of Baker domains \cite[Theorem~3]{koenig99}. This characterisation is also valid for any simply connected Baker domain of a transcendental self-map of $\C^*$.

\begin{lem}
Let $U$ be a $p$-periodic Baker domain of a meromorphic function~$f$ in which $f^{np}\to\infty$ and on which $f^p$ has a conformal conjugacy. For $z_0\in U$, put
$$
c_n=c_n(z_0):=\frac{|f^{(n+1)p}(z_0)-f^{np}(z_0)|}{\textup{dist}(f^{np}(z_0),\partial U)}.
$$
Then exactly one of the following cases holds:
\begin{enumerate}
\item[(a)] $U$ is hyperbolic and $f^p\sim T(z)=\lambda z$ with $\lambda>1$, which is \mbox{equivalent~to}
$$
c_n>c\quad \mbox{ for } z_0\in U,\ n\in \N,\quad \mbox{ where } c=c(f)>0.
$$
\item[(b)] $U$ is simply parabolic and $f^p\sim T(z)=z\pm i$, which is equivalent to 
$$
\liminf_{n\to\infty} c_n>0 \quad \mbox{ for } z_0\in U,\quad \mbox{ but } \inf_{z_0\in U}\limsup_{n\to\infty} c_n=0;
$$
\item[(c)] $U$ is doubly parabolic and $f^p\sim T(z)=z+1$, which is equivalent to
$$
\lim_{n\to\infty} c_n=0\quad \mbox{ for } z_0\in U.
$$
\end{enumerate}
\label{lem:bd-koenig}
\end{lem}

\begin{figure}[h!]
\centering
\def\svgwidth{\linewidth}
\input{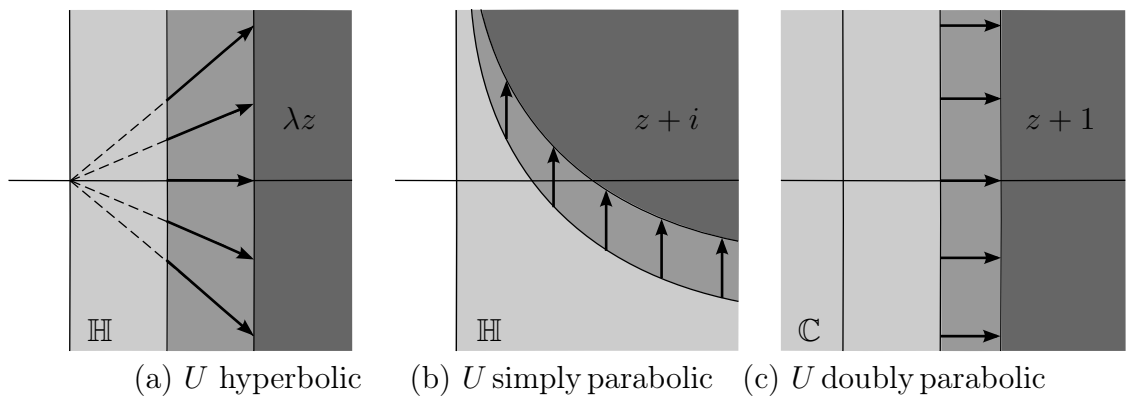}
\hspace*{17pt}(a) $U$ hyperbolic\hspace*{14pt} (b) $U$\! simply\! parabolic \hspace*{3pt} (c) $U$\! doubly\! parabolic
\caption[Classification of Baker domains with their absorbing domains]{Classification of Baker domains with their absorbing domains.}
\label{fig:bd-types}
\end{figure}

We now give a couple of explicit examples of transcendental self-maps of $\C^*$, with a hyperbolic and a doubly parabolic Baker domain, respectively.

\begin{ex}
For every $\lambda>1$, the function $f_\lambda(z)=\lambda z\exp(e^{-z}+1/z)$ is a transcendental self-map of $\C^*$ which has an invariant, simply connected, hyperbolic Baker domain $U\subseteq \C^*\setminus \R_-$ whose boundary contains both zero and infinity, and the points in~$U$ escape to infinity (see Figure~\ref{fig:hyp-baker-domain}). 
\label{ex:hyp-baker-domain}
\end{ex}

\begin{proof}[Proof of Example~\ref{ex:hyp-baker-domain}]
First observe that 
\begin{equation}
\begin{array}{rl}
f_\lambda(z)\hspace*{-8pt} & =\lambda z\exp\left(e^{-z}+\tfrac{1}{z}\right)\vspace{5pt}\\
& = \lambda z\left(1+e^{-z}+\tfrac{1}{2!}e^{-2z}+\cdots\right)\left(1+\tfrac{1}{z}+\tfrac{1}{2!}\tfrac{1}{z^2}+\cdots   \right)\vspace{5pt}\\
& = \lambda z\left(1+O\left(\tfrac{1}{z}\right)\right) \mbox{ as } \textup{Re}\,z\to\infty.
\end{array}
\label{eq:ex-bd-1}
\end{equation}
Hence $f_\lambda$ maps $\mathbb H_R:=\{z\in\C\ :\ \textup{Re}\,z>R\}$ into itself, for $R>0$ sufficiently large, so $\mathbb H_R\subseteq U$, where $U$ is an invariant Fatou component of $f_\lambda$. Also, for real $x>0$,
$$
f_\lambda(x)=\lambda x\exp\left(e^{-x}+\tfrac{1}{x}\right)>\lambda x>x
$$
so $f_\lambda^n(x)\to\infty$ as $n\to \infty$. Thus, $U$ is an invariant Baker domain of $f$ which contains the positive real axis, so $\partial U$ contains zero and infinity.

To show that $U$ is a hyperbolic Baker domain, consider $z_0\in U$. By the contraction property of the hyperbolic metric in $U$, the orbit of $z_0$ escapes to infinity in~$\mathbb H_R$. Hence, by \eqref{eq:ex-bd-1} and since $0\in U^c$,
$$
\begin{array}{rl}
c_n\hspace*{-6pt} &\ds=\frac{|f^{n+1}(z_0)-f^n(z_0)|}{\textup{dist}\,(f^n(z_0),\partial U)}\geqslant \frac{\lambda f^n(z_0)\left(1+O\left(\frac{1}{f^n(z_0)}\right)\right)-f^n(z_0)}{|f^n(z_0)|}\vspace{10pt}\\
&\ds>\lambda -1 -\frac{O(1)}{|f^n(z_0)|} \mbox{ as } n\to \infty,
\end{array}
$$
so
$$
\liminf_{n\to\infty} c_n\geqslant \lambda-1>0,
$$
and thus the Baker domain $U$ is hyperbolic.

Finally, observe that the negative real axis is invariant under $f$, and therefore $(-\infty,0)\cap U=\emptyset$. Since doubly connected Fatou components must surround zero, $U$ is simply connected.
\end{proof}

\begin{figure}[h!]
\includegraphics[width=.49\linewidth]{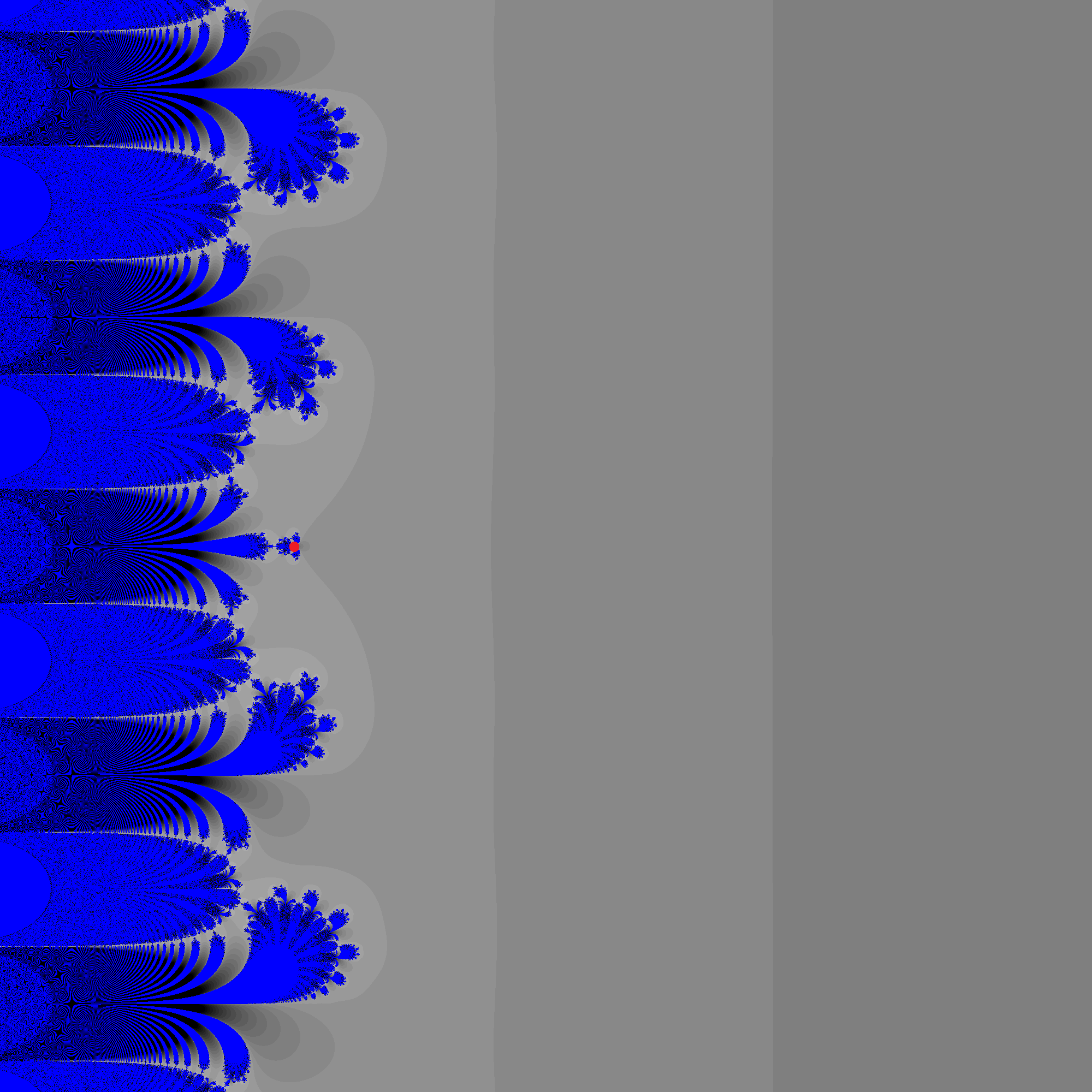}
\includegraphics[width=.49\linewidth]{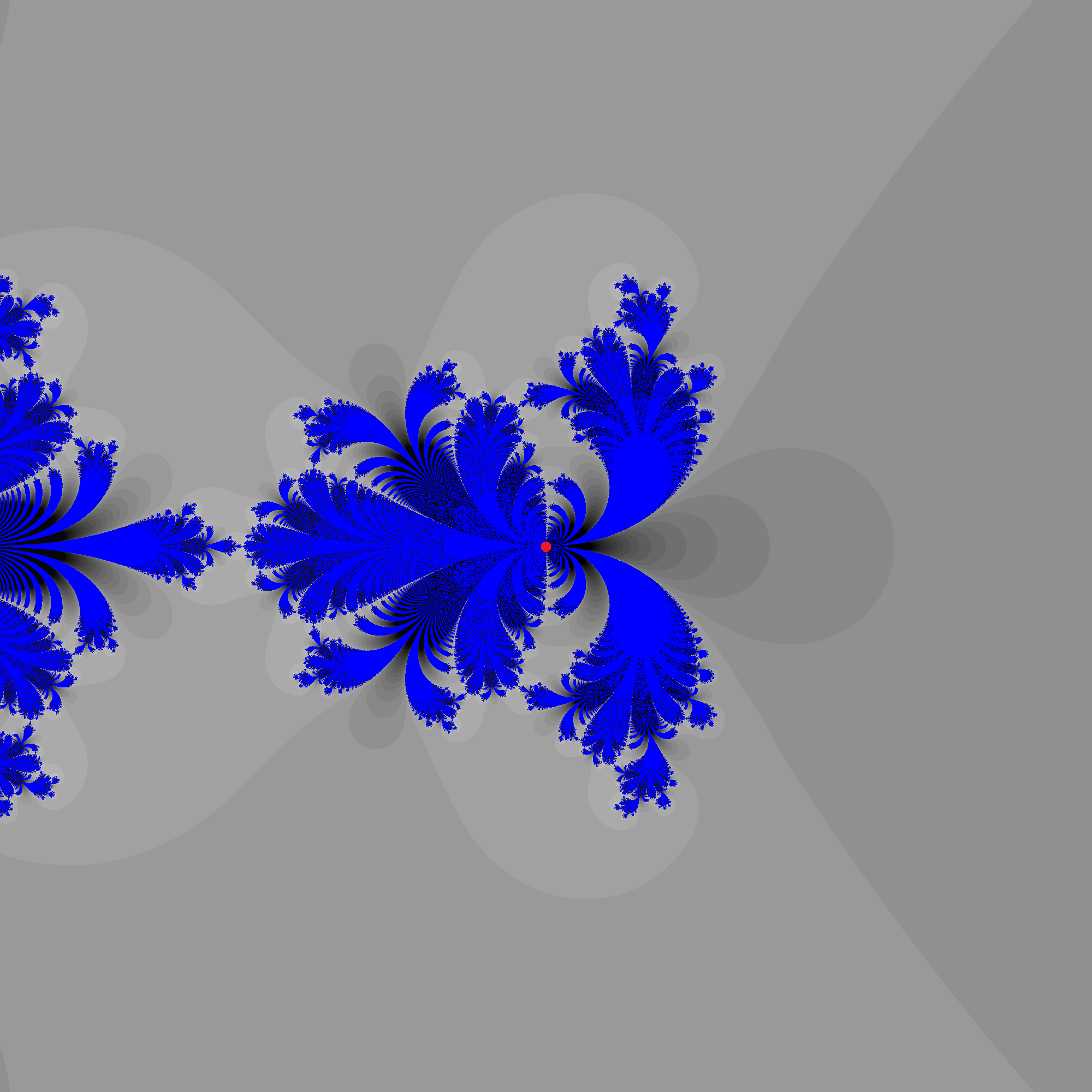}
\caption[Phase space of a transcendental self-map of $\C^*$ which has a hyperbolic Baker domain]{Phase space of the function $f_2(z)=2z\exp(e^{-z}+1/z)$ from Example~\ref{ex:hyp-baker-domain}. On the right, zoom of a neighbourhood of zero.}
\label{fig:hyp-baker-domain}
\end{figure}

The function $f(z)=2 z\exp(e^{-z}+1/z)$ has a repelling fixed point in the negative real line. If we choose $h(z)=1/z^2$ instead of $1/z$, then $f(z)=2 z\exp(e^{-z}+1/z^2)$ has the positive real axis in a Baker domain while the negative real axis is in the fast escaping set.



We now give a second explicit example of transcendental self-map of $\C^*$ with a Baker domain which, in this case, is doubly parabolic.

\begin{ex}
The function $f(z)=z\exp\left((e^{-z}+1)/z\right)$ is a transcendental self-map of $\C^*$ which has an invariant, simply connected, doubly parabolic Baker domain $U\subseteq \C^*\setminus \R_-$ whose boundary contains both zero and infinity, and the points in $U$ escape to infinity (see Figure~\ref{fig:dpar-baker-domain}). \vspace*{-10pt}
\label{ex:dpar-baker-domain}
\end{ex}

\begin{proof}[Proof of Example~\ref{ex:dpar-baker-domain}]
Looking at the power series expansion of $f$, we have 
$$
\begin{array}{rl}
f(z)\hspace*{-8pt} & = z\exp\left(\tfrac{e^{-z}}{z}+\tfrac{1}{z}\right)\vspace{5pt}\\
& = z\left(1+\tfrac{e^{-z}}{z}+\tfrac{1}{2!}\tfrac{e^{-2z}}{z^2}+\cdots\right)\left(1+\tfrac{1}{z}+\tfrac{1}{2!}\tfrac{1}{z^2}+\cdots   \right)\vspace{5pt}\\
& =  z\left(1+\tfrac{1}{z}+O\left(\tfrac{1}{z^2}\right)\right) \mbox{ as } \textup{Re}\,z\to\infty.
\end{array}
$$
Therefore $f$ maps the right half-plane $\mathbb H_R:=\{z\in\C\ :\ \textup{Re}\, z>R\}$ into itself for sufficiently large values of $R>0$ and $\mathbb H_R$ is contained in an invariant Baker domain $U$ of $f$, in which $\textup{Re}\,f^n(z)\to+\infty$ as $n\to\infty$. Since $f(x)>x$ for all $x>0$, the positive real axis lies in $U$. Let $z_0\in U$, then
$$
f^{n+1}(z_0)-f^n(z_0)\! =\! f^n(z_0)\! \left(\! 1+O\! \left(\! \frac{1}{f^n(z_0)}\! \right)\! \right)-f^n(z_0)\! =\! O(1)\mbox{ as } n\to \infty
$$
and, if $R$ is as above,
$$
\textup{dist}(f^n(z_0),\partial U)\geqslant \textup{Re}\,f^n(z_0)-R \quad \mbox{ as } n\to \infty,
$$
so 
$$
c_n=\frac{|f^{n+1}(z_0)-f^{n}(z_0)|}{\textup{dist}(f^{n}(z_0),\partial U)}\leqslant \frac{O(1)}{\textup{Re}\,f^n(z_0)-R}\to 0 \quad \mbox{ as } n\to\infty.
$$
Thus, by Lemma~\ref{lem:bd-koenig}, the Baker domain $U$ is doubly parabolic.

Finally, observe that, for $x\in (-\infty,0)$, $f^n(x)\to \infty$ along the negative real axis as $n\to\infty$, so $(-\infty,0)\cap U=\emptyset$ and hence $U$ is simply connected.
\end{proof}

\begin{figure}[h!]
\includegraphics[width=.49\linewidth]{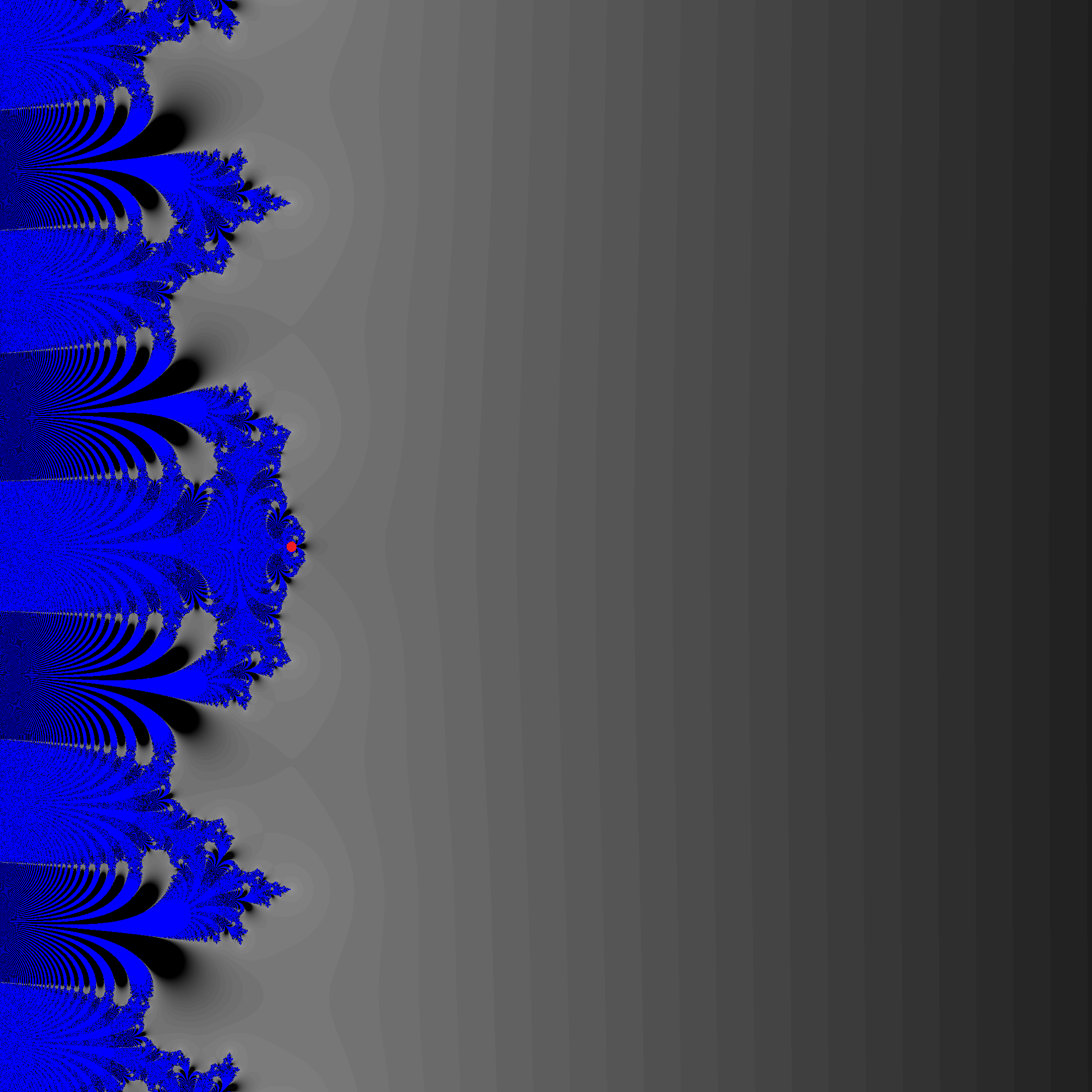}
\includegraphics[width=.49\linewidth]{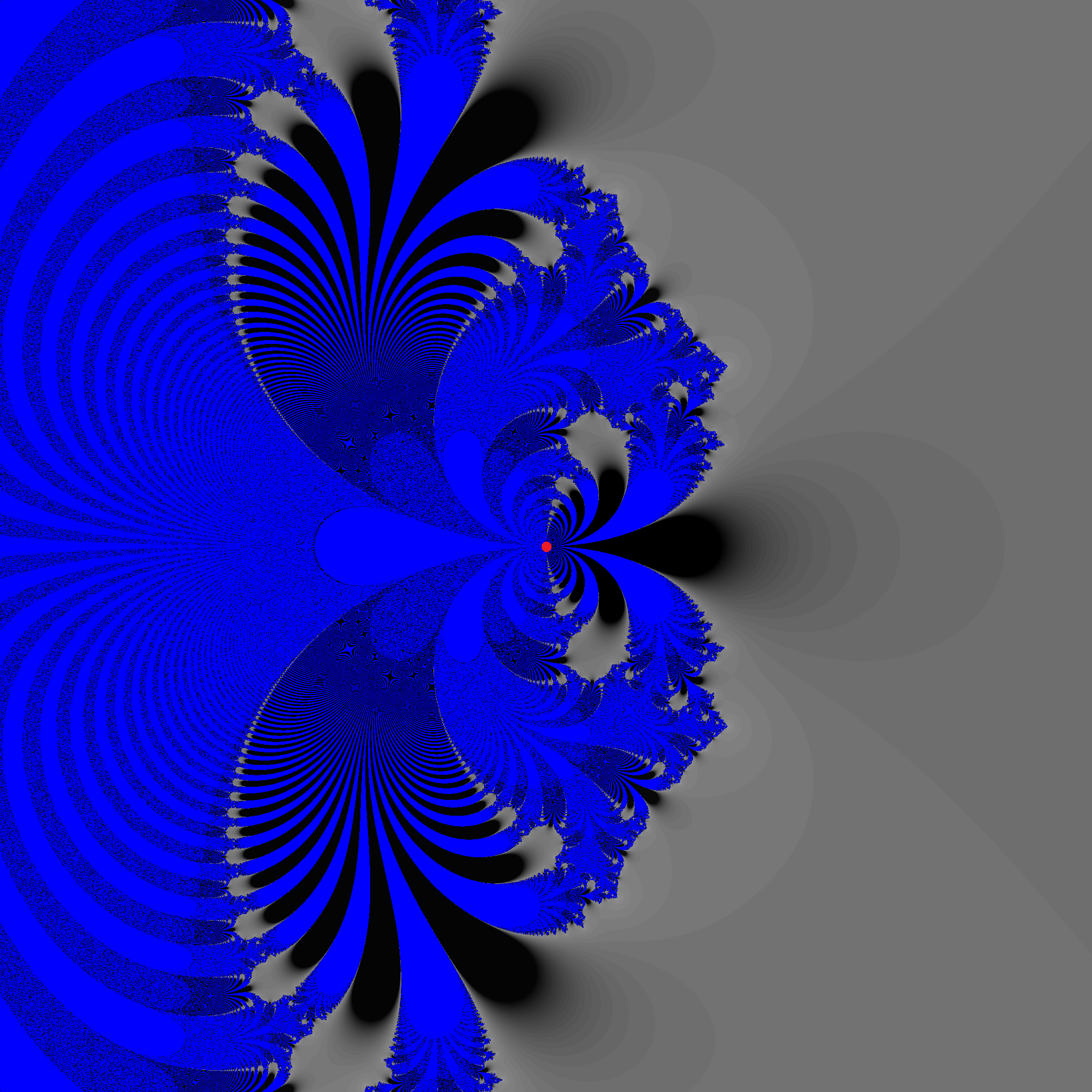}
\caption[\mbox{Phase space of a transcendental self-map of $\C^*$} \mbox{which has a doubly parabolic Baker domain}]{Phase space of the function $f(z)=z\exp\left((e^{-z}+1)/z\right)$ from Example~\ref{ex:dpar-baker-domain}. On the right, zoom of a neighbourhood of zero.}
\label{fig:dpar-baker-domain}
\end{figure}


\begin{lem}
Let $f$ be a transcendental self-map of $\C^*$ and let $\tilde{f}$ be a lift of $f$. Then, if $U$ is a Baker domain of $f$, every component $U_k,\ k\in\Z,$ of $\exp^{-1}(U)$ is either a (preimage of a) Baker domain or a wandering domain \mbox{of~$\tilde{f}$}. Moreover, if $U$ is simply connected and $U_k$ is a Baker domain, then $U_k$ is hyperbolic, simply parabolic or doubly parabolic if and only if $U$ is hyperbolic, simply parabolic or doubly parabolic, respectively.
\label{lem:semiconj-bd}
\end{lem}
\begin{proof}
By \cite{bergweiler95}, every component of $\exp^{-1}(U)$ is a Fatou component of $\tilde{f}$. Moreover, since $\exp^{-1}(I(f))\subseteq I(\tilde{f})$, $U_k$ is either a Baker domain, a preimage of a Baker domain or an escaping wandering domain of $\tilde{f}$.

Suppose that $U$ has period $p\geqslant 1$ and $U_k$ is periodic. Then the Baker domain $U_k$ has period $q$ with $p\mid q$. Let $(V,\phi,T)$ be a conformal conjugacy of $f^q$ in $U$. Then $(\tilde{V},\tilde{\phi},T)$ is a conformal conjugacy of $\tilde{f}^q$ in $U_k$, where $\tilde{V}$ is the component of $\exp^{-1}V$ that lies in $U_k$ and $\tilde{\phi}=\phi\circ\exp$. Thus, the Baker domains $U$ and $U_k$ are of the same type.
\end{proof}

As before, we use Lemma~\ref{lem:semiconj-bd} to provide examples of transcendental entire functions with Baker domains and wandering domains.

\begin{ex}
The entire function $\tilde{f}(z)\!=\!\ln \lambda\!+\!z\!+\!\exp(-e^z)\!+\!e^{-z}$, which is a lift of the function $f$ from Example~\ref{ex:hyp-baker-domain}, has an invariant hyperbolic Baker domain that contains the real line.
\end{ex}

\begin{ex}
The entire function $\tilde{f}(z)=z+\frac{\exp(-e^z)}{e^z}+e^{-z}$, which is a lift of the function $f$ from Example~\ref{ex:dpar-baker-domain}, has an invariant doubly parabolic Baker domain that contains the real line.
\end{ex}

%
%
%
%
%
%

\section{Preliminaries on approximation theory} \label{sec:approx-theory}

In this section we state the results from approximation theory that will be used in Sections~\ref{sec:wd} and \ref{sec:bd} to construct examples of functions with wandering domains and Baker domains, respectively. We follow the terminology from \cite[Chapter~IV]{gaier87}, and introduce Weierstrass and Carleman sets. Recall that if $F\subseteq \C$ is a closed set, then $A(F)$ denotes the set of continuous functions $f:F\to\C$ that are holomorphic in the interior of $F$.



\begin{dfn}[Weierstrass set]
\label{dfn:weierstrass-set}
We say that a closed set $F\subseteq\C $ is a \textit{Weierstrass set} in $\C$ if each $f\in A(F)$ can be approximated by entire functions \textit{uniformly} on $F$; that is, for every $\varepsilon>0$, there is an entire function $g$ for which
$$
|f(z)-g(z)|<\varepsilon\quad \mbox{ for all } z\in F.
$$
\end{dfn}

The next result is due to Arakelyan and provides a characterisation of Weierstrass sets \cite{arakeljan64}. In the case that $F\subseteq \C$ is compact and $\C\setminus F$ is connected, then it follows from Mergelyan's theorem \cite[Theorem~1~on~p.~97]{gaier87} that functions in $A(F)$ can be uniformly approximated on $F$ by polynomials.

\begin{lem}[Arakelyan's theorem]
A closed set $F\subseteq \C$ is a Weierstrass set if and only if the following two conditions are satisfied:
\begin{enumerate}
\item[\emph{(K$_1$)}] $\CR\setminus F$ is connected;
\item[\emph{(K$_2$)}] $\CR\setminus F$ is locally connected at infinity.
\end{enumerate}
\end{lem}


If in addition both the set $F$ and the function $f\in A(f)$ are symmetric with respect to the real line, then the approximating function $g$ can be chosen to be symmetric as well (see \cite[Section 2]{gauthier13}).

Sometimes we may want to approximate a function in $A(f)$ so that the error is bounded by a given strictly positive function $\varepsilon:\C\to\R_+$ that is not constant, and $\varepsilon(z)$ may tend to zero as $z\to\infty$. 

\begin{dfn}[Carleman set]
\label{dfn:carleman-set}
We say that a closed set $F\subseteq \C$ is a \textit{Carleman set} in $\C$ if every function $f\in A(F)$ admits \textit{tangential approximation} on $F$ by entire functions; that is, for every strictly positive function $\varepsilon\in\mathcal C(F)$, there is an entire function $g$ for which
$$
|f(z)-g(z)|<\varepsilon(z)\quad \mbox{ for all } z\in F.
$$
\end{dfn}

It is clear that Carleman sets are a special case of Weierstrass sets and hence conditions ($\text{K}_1$) and ($\text{K}_2$) are necessary. Nersesyan's theorem gives sufficient conditions for tangential approximation \cite{nersesjan71}.

\begin{lem}[Nersesyan's theorem]
A closed set $F$ is a Carleman set in~$\C$ if and only if conditions \emph{($K_1$)}, \emph{($K_2$)} and 
\begin{enumerate}
\item[\emph{(A)}] for every compact set $K\subseteq \C$ there exists a neighbourhood $V$ of infinity in $\CR$ such that no component of $\mbox{int}\, F$ intersects both $K$ and $V$,
\end{enumerate}
are satisfied.
\label{lem:nersesjan}
\end{lem}

Note that there is also a symmetric version of this result: if the set $F$ and the functions $f$ and $\varepsilon$ are in addition symmetric with respect to $\R$ then the entire function $g$ can be chosen to be symmetric with respect to $\R$ \cite[Section 2]{gauthier13}.

In some cases, depending on the geometry of the set $F$ and the decay of the error function $\varepsilon$, we can perform tangential approximation on Weierstrass sets without needing condition (A); the next result can be found in \cite[Corollary in p.162]{gaier87}.

\begin{lem}
Suppose $F\subseteq \C$ is a closed set satisfying conditions ($\text{K}_1$) and ($\text{K}_2$) that lies in a sector
$$
W_\alpha:=\{z\in \C\ :\ |\textup{arg}\,z|\leqslant \alpha/2\},
$$
for some $0<\alpha\leqslant 2\pi$. Suppose $\tilde{\varepsilon}(t)$ is a real function that is continuous and positive for $t\geqslant 0$ and satisfies
$$
\int_1^{+\infty} t^{-(\pi/\alpha)-1}\log\tilde{\varepsilon}(t)dt>-\infty.
$$
Then every function $f\in A(F)$ admits $\varepsilon$-approximation on the set $F$ with $\varepsilon(z)=\tilde{\varepsilon}(|z|)$ for $z\in F$. 
\label{lem:approx-sectors}
\end{lem}

\section{Construction of functions with wandering domains}

\label{sec:wd}


To prove Theorem \ref{thm:wandering-domains} we modify Baker's construction of a holomorphic self-map of $\C^*$ with a wandering domain escaping to infinity \cite[Theorem 4]{baker87} to create instead a transcendental self-map of $\C^*$ with a wandering domain that accumulates to zero and to infinity according to a prescribed essential itinerary $e\in\{0,\infty\}^{\N_0}$ and with index $n\in \Z$.

\begin{proof}[Proof of Theorem \ref{thm:wandering-domains}] 
We construct two entire functions $g$ and $h$ using Nersesyan's theorem so that the function $f(z)=z^n\exp\bigl(g(z)+h(1/z)\bigr)$, which is a transcendental self-map of $\C^*$, has the following properties:
\begin{itemize}
\item there is a bi-infinite sequence of annuli sectors $\{A_m\}_{m\in\Z\setminus\{0\}}$ that accumulate at zero and infinity and integers $s(m)\in\Z\setminus\{0\}$, for $m\in\Z\setminus\{0\}$, such that $f(A_m)\subseteq A_{s(m)}$ for all $m\in\Z$;
\item the discs $B_+:=\overline{D(2,1/4)}$ and $B_-:=1/B_+=\overline{D(32/63, 4/63)}$ both map strictly inside themselves under $f$, $f(B_+)\subseteq \textup{int}\,B_+$ and $f(B_-)\subseteq \textup{int}\,B_-$;
\item there is a bi-infinite sequence of closed discs $\{B_m\}_{m\in\Z\setminus\{0\}}$ such that $f(B_m)\subseteq \textup{int}\,B_+$, if $m>0$, and $f(B_m)\subseteq \textup{int}\,B_-$, if $m<0$.
\end{itemize}
Here $s(m):=\pi(\pi^{-1}(m)+1)$ and the map $\pi:\N\longrightarrow \Z\setminus \{0\}$ is an ordering of the sets $\{A_m\}_{m\in\N}$ according to the sequence $e$; that is, $\pi(k)$ is the position of the $k$th component in the orbit of the wandering domain. More formally, we define
\begin{equation}
\pi(k):=\left\{
\begin{array}{ll}
\ds\#\{\ell\in\N_0\ :\ e_\ell=\infty \mbox{ for } \ell<k\}+1, & \mbox{ if } e_{k}=\infty,\vspace{5pt}\\
\ds-\,\#\{\ell\in\N_0\ :\ e_\ell=0 \mbox{ for } \ell<k\}-1, & \mbox{ if } e_{k}=0,
\end{array}
\right.
\label{eq:p-function}
\end{equation}
for $k\in\N$ (see Figure \ref{fig:sketch-wd}).

\begin{figure}[h!]
\centering
\vspace*{45pt}
\def\svgwidth{.8\linewidth}
\input{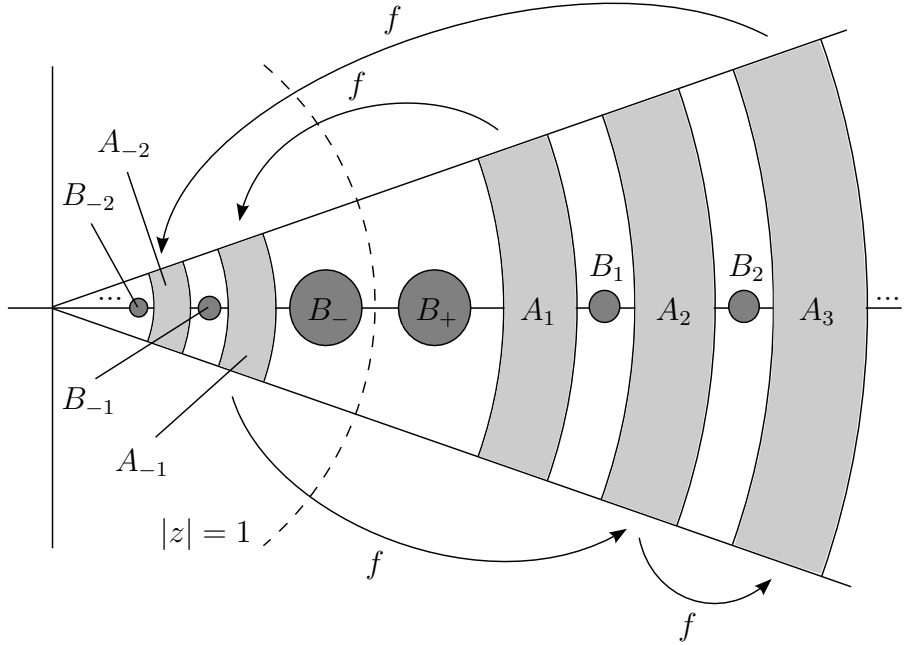}
\vspace*{30pt}
\caption[Sketch~of~the~construction~of~a~transcendental \mbox{self-map\,of\,$\C^*$\,with\,a\,wandering\,domain}]{Sketch of the construction in the proof of Theorem \ref{thm:wandering-domains}.}
\label{fig:sketch-wd}
\end{figure}

By Montel's theorem, the domains $\{A_m\}_{m\in \Z\setminus\{0\}}$, $\{B_m\}_{m\in\Z\setminus\{0\}}$ and $B_+,B_-$ are all contained in the Fatou set. Since $f(B_+)\subseteq \textup{int}\,B_+$, the function $f$~has an attracting fixed point in $B_+$ and the sets $\{B_m\}_{m\in\N}$ are contained in the preimages of the immediate basin of attraction of this fixed point. Likewise, the sets $\{B_{-m}\}_{m\in\N}$ belong to the basin of attraction of an attracting fixed point in $B_-$. Observe that in order to show that $A_1$ is contained in a wandering domain that escapes following the essential itinerary $e$ we need to prove that every $A_m$ is contained in a different Fatou component. 

Now let us construct the entire functions $g$ and $h$ so that the function $f(z)=z^n\exp\bigl(g(z)+h(1/z)\bigr)$ has the properties stated above. Note that in this construction $\log z$ denotes the principal branch of the logarithm with \mbox{$-\pi<\textup{arg}\,z<\pi$}. Let $0<R<\pi/2$ and set, for $m>0$, define
$$
\begin{array}{l}
A_m:=\{z\in\C\ :\ -R\leqslant\mbox{arg}(z)\leqslant R,\ k_m\leqslant |z|\leqslant k_m e^{2R}\},\vspace{5pt}\\
B_m:=\overline{D\bigl((k_{m+1}-k_m)/2,\ 1/8\bigr)},
\end{array}
$$
where $k_m$ is any sequence of positive real numbers such that \mbox{$k_m\! >5/2$} and $k_{m+1}>k_m+1/4$ for all $m\in \N$. We define $A_{-m}:=1/A_m$ and $B_{-m}:=1/B_m$ for all $m\in\N$. Note that $\log A_m$ is a square of side $2R$ centred at a point that we denote by $a_m\in \R$. Hence, $\log A_m$ contains the disc $D(a_m,R)$ for all $m\in\Z\setminus\{0\}$. The set 
$$
F:=\overline{D(0,1)}\cup B_+\cup \bigcup_{m>0} (A_m\cup B_m)
$$
which consists of a countable union of disjoint compact sets is a Carleman set.

Let $\delta_+,\delta_->0$ be such that $|w-\ln 2|\! <\! \delta_+$ and \mbox{$|w-\ln 32/63|\! <\! \delta_-$} imply, respectively, that $|e^w-2|<1/8$ and $|e^w-32/63|<2/63$. Let $K:=\min\{R/4, \delta_\pm/4\}$. By Lemma \ref{lem:nersesjan}, there is an entire function $g$ that satisfies the following conditions:
$$
\left\{
\begin{array}{ll}
|g(z)-a_{s(m)}-n\log z|<R/4, & \mbox{if } z\in A_m \mbox{ with } m>0,\vspace{10pt}\\
|g(z)-\ln 2-n\log z|<\delta_+/4, & \ds\mbox{if } z\in \bigcup_{m>0} B_m\cup B_+,\vspace{5pt}\\
|g(z)|<K, & \mbox{if } z\in D(0,1),
\end{array}
\right.\vspace{5pt}
$$
Similarly, there is an entire function $h$ that satisfies the following conditions:
 $$ \left\{
\begin{array}{ll}
|h(z)-a_{s(-m)}-n\log (1/z)|<R/4, & \mbox{if } z\in A_m \mbox{ with } m>0,\vspace{10pt}\\
|h(z)-\ln 32/63-n\log (1/z)|<\delta_-/4, & \ds\mbox{if } z\in \bigcup_{m>0} B_m\cup B_+,\vspace{5pt}\\
|h(z)|<K, & \mbox{if } z\in D(0,1).
\end{array}
\right.\vspace{5pt}
$$
Therefore, since the sets $B_-$ and $A_m$, $m<0$, are contained in $D(0,1)$ and the sets $B_+$ and $A_m$, $m>0$, are contained in $\C\setminus \overline{D(0,1)}$, the function $\log f(z)=g(z)+h(1/z)+n\log z$ satisfies
$$
\left\{
\begin{array}{ll}
|\log f(z)-a_{s(m)}|<R/2, & \mbox{if } z\in A_m \mbox{ with } m\neq 0,\vspace{10pt}\\
|\log f(z)-\ln 2|<\delta_+/2, & \ds\mbox{if } z\in  \bigcup_{m>0} B_m\cup B_+,\vspace{10pt}\\
|\log f(z)-\ln 32/63|<\delta_-/2, & \ds\mbox{if } z\in  \bigcup_{m<0} B_m\cup B_-,\\
\end{array}
\right.
$$
and hence $f$ has the required mapping properties.


Finally, note that this construction is symmetric with respect to the real line and hence all Fatou components of $f$ that intersect the real line will be symmetric too. Thus, since transcendental self-maps of~$\C^*$ cannot have doubly connected Fatou components that do not surround the origin \cite[Theorem 1]{baker87}, the Fatou components containing the sets $\{A_m\}_{m\in\Z\setminus \{0\}}$ are pairwise disjoint and $A_{\pi(0)}$ is contained in a wandering domain in $I_e(f)$.
\end{proof}


\section{Construction of functions with Baker domains}

\label{sec:bd}

In this section we construct holomorphic self-maps of $\C^*$ with Baker domains. The construction is split into two cases: first, we deal with the cases that the function $f$ is a transcendental entire or meromorphic function, that is, $f(z)=z^n\exp(g(z))$ where $n\in\Z$ and $g$ is a non-constant entire function (see Theorem~\ref{thm:baker-domains-entire}), and then we deal with the case that the function $f$ is a transcendental self-map of~$\C^*$, that is, $f(z)=z^n\exp(g(z)+h(1/z))$ where $n\in \Z$ and $g,h$ are non-constant entire functions (see Theorem~\ref{thm:baker-domains}). For transcendental self-maps of~$\C^*$, we are able to construct functions with Baker domains that have any given \textit{periodic} essential itinerary $e\in\{0,\infty\}^{\N_0}$. 

To that end, we use Lemma~\ref{lem:approx-sectors}~to obtain entire functions $g$ and, if necessary, $h$ so that the function $f$ has a Baker domain. After this approximation process, the resulting function $f$ will behave as the function~$T_\lambda(z)=\lambda z$, $\lambda>1$, in a certain half-plane~$W$. We first require the following result that estimates~the asymptotic distance between the boundaries of $\log W$ and $\log T_\lambda(W)\subseteq \log W$.

\begin{lem}
Let $W=\{z\in\C : \textup{Re}\, z\geqslant 2\}$ and, for $\lambda>1$, let \mbox{$T_\lambda(z)=\lambda z$}. For $r>0$, let $\delta (r)$ denote the vertical distance between the curves $\partial\log  W$ and $\partial\log T_\lambda(W)\subseteq \log W$ along the vertical line $V_r:=\{z\in\C : \textup{Re}\, z=r\}$. Then $\delta(r)\sim 2(\lambda-1)e^{-r}$ as $r\to+\infty$.
\label{lem:approx-BD}
\end{lem}
\begin{proof}
Since $\log z=\ln |z|+i\,\mbox{arg}(z)$, the quantity $\delta(r)$ equals the difference between the arguments of the points $z_1,z_2$ with $\textup{Im}\, z_k>0$, $k\in\{1,2\}$, where the vertical lines $\partial W$ and $\partial T(W)$ intersect the circle $\exp V_r$ of radius $e^r$ (see Figure \ref{fig:sketch-bd}).  

\begin{figure}[h!]
\centering
\vspace{30pt}
\def\svgwidth{.8\linewidth}
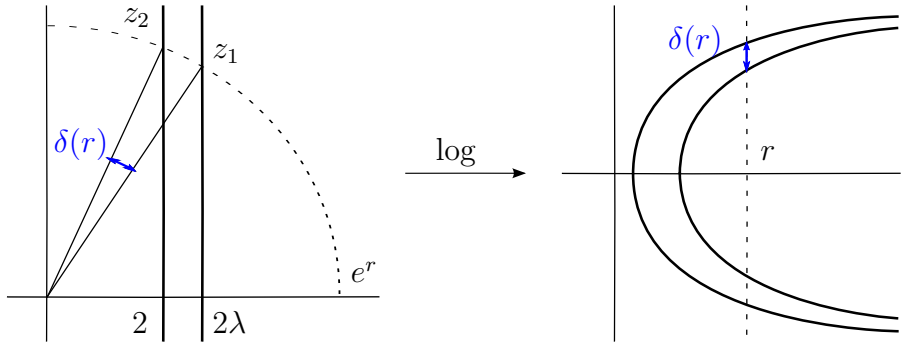
\vspace{15pt}
\caption[Definition of the function $\delta(r)$]{Definition of the function $\delta(r)$.}
\label{fig:sketch-bd} 
\end{figure}

Since $\mbox{arg}\,z_1,\mbox{arg}\,z_2 \to\pi/2$ as $r\to+\infty$, we have
$$
\delta(r)=\arccos \frac{2}{e^r}-\arccos\frac{2\lambda}{e^r}\sim\left(\frac{\pi}{2}-\frac{2}{e^r}\right)-\left(\frac{\pi}{2}-\frac{2\lambda}{e^r}\right)=\frac{2(\lambda-1)}{e^r},
$$
as $r\to+\infty$, as required.
\end{proof}

Given $N\in\N$ and a periodic sequence \mbox{$e=\overline{e_0e_1\cdots e_{N-1}}\in\{0,\infty\}^{\N_0}$}, let $p,q\in\N$ denote
\begin{equation}
\begin{array}{l}
p=p(e):=\#\{k\in\N_0\ :\ e_k=\infty \mbox{ for } k<N\},\vspace{5pt}\\
q=q(e):=\#\{k\in\N_0\ :\ e_k=0 \mbox{ for } k<N\},
\end{array}
\label{eq:p-and-q}
\end{equation}
so that $p+q=N$. We want to construct a holomorphic function \mbox{$f:\C^*\to\C^*$} with an $N$-cycle of Baker domains that has components $U_i^\infty$, $0\leqslant i<p$, and $U_i^0$, $0\leqslant i<q$, in which 
$$
f_{|U_i^\infty}^{Nn}\to\infty\quad \mbox{ and } \quad f_{|U_i^0}^{Nn}\to 0 \quad \mbox{ locally uniformly as } n\to \infty.
$$
In the case that zero is \textit{not} an essential singularity of $f$, then $q=0 $ and $N=p$. Note that the closure of a Baker domain in $\CR$ may contain both zero and infinity. 

For $p\in\N$ and $X\subseteq \C^*$, we define
$$
\sqrt[p]{X}:=\{z\in\C^*\, :\, z^p\in X,\ |\textup{arg}\,z|<\pi/p\}.
$$
In order to construct a function with an $N$-periodic Baker domain that has $p$ components around zero or infinity, we will semiconjugate the function $T_\lambda$ that we want to approximate in the half-plane $W$ by the $p$th root function:
$$
\xymatrix{
W \ar[r]^{T_\lambda}  & W\\
\sqrt[p]{W} \ar[u]^{z^p} \ar[r]_{T_{\lambda,p}} & \sqrt[p]{W}. \ar[u]_{z^p}
}
$$
Next we look at the effect of this semiconjugation on the function $\delta$.

\begin{lem}
Let $W$ and $T_\lambda$, $\lambda>1$, be as in Lemma~\ref{lem:approx-BD}. For $p\in\N$ and $\lambda>1$, define the function $T_{\lambda,p}(z):=\sqrt[p]{T_\lambda(z^p)}$ on $\sqrt[p]{W}$ and, for $r>0$, let $\delta_{p}(r)$ denote the vertical distance between the curves $\partial \log \sqrt[p]{W}$ and $\partial \log T_{\lambda,p}(\sqrt[p]{W})\subseteq \log \sqrt[p]{W}$ along the vertical line \mbox{$V_r:=\{z\in\C : \textup{Re}\, z=r\}$}. Then $\delta_{p}(r)\sim 2(\lambda-1)e^{-pr}/p$ as $r\to+\infty$.
\label{lem:approx-root-BD}
\end{lem}
\begin{proof}
The function $z\mapsto z^p$ maps the circle of radius $e^r$ to the circle of radius $e^{pr}$ while the function $z\mapsto \sqrt[p]{z}$ divides the argument of points on that circle by $p$, so
$$
\delta_{p}(r)=\frac{\delta(pr)}{p}
$$
and hence, by Lemma~\ref{lem:approx-BD}, $\delta_{p}(r)\sim 2(\lambda-1)e^{-pr}/p$ as $r\to+\infty$.
\end{proof}

In the following theorem we construct transcendental entire or meromorphic functions that are self-maps of $\C^*$ and have Baker domains in which points escape to infinity. These functions are of the form $f(z)=z^n\exp(g(z))$ where $n\in\Z$ and $g$~is a non-constant entire function. 

\begin{thm}
\label{thm:baker-domains-entire}
For every $N\in\N$ and $n\in\Z$, there exists a holomorphic self-map~$f$ of $\C^*$ with $\textup{ind}(f)=n$ that is a transcendental entire function, if $n\geqslant 0$, or a transcendental meromorphic function, if $n<0$, and has a cycle of hyperbolic Baker domains of period $N$.
\end{thm}
\begin{proof}
Let $\omega_{N}:=e^{2\pi i/N}$ and define 
$$
V_{m}:=\omega_{N}^m\sqrt[N]{W}\subseteq \C\setminus \overline{\mathbb D} \quad \mbox{ for } 0\leqslant m<N,
$$
where $W$ is the closed half-plane from Lemma~\ref{lem:approx-BD}. We denote by $V$ the union of all $V_m$ for $0\leqslant m<N$, and let $R:=\R_-$, if $N$ is odd, or $R:=\{z\in\C^*\,:\,\textup{arg}\,z=\pi(1-1/N)\}$, if $N$ is even. Then put 
\begin{equation}
d:=\min\{(\sqrt[N]{2}-1)/3,\ \textup{dist}(V,R)/4\},
\label{eq:bd-1}
\end{equation}
and define the closed connected set
\begin{equation}
B:=\{z\in\C\,:\,\textup{dist}\,(z,V)\geqslant d \mbox{ and } \textup{dist}\,(z,R)\geqslant d\},
\label{eq:bd-2}
\end{equation}
which satisfies $B':=\overline{D(1,d)}\subseteq \textup{int}\,B$ (see Figure~\ref{fig:sketch-bd-entire}).

\begin{figure}[h!]
\centering
\def\svgwidth{.60\linewidth}
\input{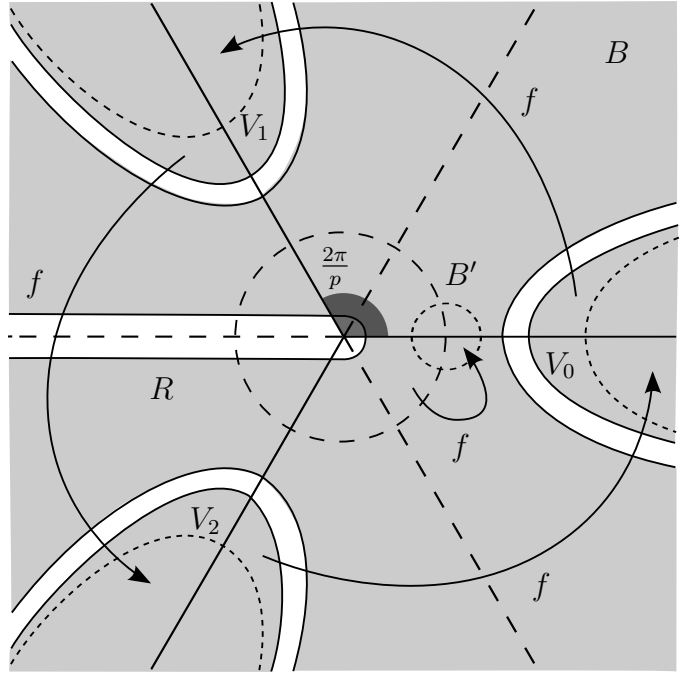}
\caption[Sketch of the construction of a transcendental entire or meromorphic function that is a self-map of $\C^*$ and has a cycle of hyperbolic Baker domains]{Sketch of the construction in the proof of Theorem~\ref{thm:baker-domains-entire} with \mbox{$N\! =\! 3$}. The sets $B$ and $V_m$, $0\leqslant m<N$, are shaded in grey.}
\label{fig:sketch-bd-entire} 
\end{figure}

Observe that the closed set $F:=B\cup V$ satisfies the hypothesis of Lemma \ref{lem:approx-sectors}; namely $\CR\setminus F$ is connected and $\CR \setminus F$ is locally connected at infinity, and $F\subseteq W_\alpha$ with $\alpha=2\pi$. We now define a function $\hat{g}$ on~$F$:
\begin{equation}
\hat{g}(z)\!:=\!\left\{\begin{array}{ll}
\!\!\!\log \left(\omega_{N}^{m+1}\sqrt[N]{\lambda (z/\omega_{N}^m)^{N}}\right)\! -n\log z, & \!\mbox{for } z\in V_m,\ 0\leqslant m<N,\vspace{5pt}\\
\!\!\!-n\log z, & \!\mbox{for } z\in B,
\end{array}\right. 
\label{eq:bd-3}
\end{equation}
where we have taken an analytic branch of the logarithm defined on $\C^*\setminus R$ and hence on $F$. Then $\hat{g}\in A(F)$.

For $r>0$, we define the positive continuous function 
\begin{equation}
\varepsilon(r):=\min\{d',\ k^{-(N+1)},\ r^{-(N+1)}\} 
\label{eq:bd-4}
\end{equation}
where the constant $d'>0$ is so small that $|e^z-1|<d$ for $|z|<d'$ and the constant $k>0$ is so large that, for all $z\in \log T_\lambda(W)$ with $\textup{Re}\,z<k$, the disc $D(z,k^{-(N+1)})$ is compactly contained in $\log W$ and, moreover, if $\delta_{N}(r)$ is the function from Lemma~\ref{lem:approx-root-BD}, then
\begin{equation}
\varepsilon(r)<\delta_{N}(\ln (\lambda r)) \quad \mbox{ for } r\geqslant k,
\label{eq:bd-5}
\end{equation}
which is possible since
$$
\delta_{N}(\ln (\lambda r))\sim \frac{2(\lambda-1)}{N\lambda^N r^N} \quad \mbox{ as } r\to+\infty.
$$
Since $\varepsilon$ satisfies
$$
\int_1^{+\infty} r^{-3/2}\ln\varepsilon(r)dt=C-(N+1)\int_{r_0'}^{+\infty} \frac{\ln r}{r^{3/2}}dr>-\infty
$$
for some constants $C\in\R$ and $r_0'\geqslant r_0$, by Lemma \ref{lem:approx-sectors} (with $\alpha=2\pi$), there is an entire function~$g$ such that
\begin{equation}
|g(z)-\hat{g}(z)|<\varepsilon(|z|)\quad\mbox{ for all } z\in F.
\label{eq:bd-entire-approx}
\end{equation}
We put 
\begin{equation}
f(z):=z^n\exp(g(z))=z^n\exp(\hat{g}(z))\exp(g(z)-\hat{g}(z)).
\label{eq:bd-6}
\end{equation}
By Lemma \ref{lem:approx-root-BD} and (\ref{eq:bd-3}-\ref{eq:bd-entire-approx}), $f(V_m)\subseteq V_{m+1}$ for \mbox{$0\leqslant m<N- 1$} and $f(V_{N-1})\subseteq V_0$ and, by (\ref{eq:bd-1}-\ref{eq:bd-entire-approx}), $f(B)\subseteq D(1,d)$. Hence each set $V_m$ is contained in an $N$-periodic Fatou component~$U_m$ for $0\leqslant m<N$ and~$B$ is contained in the immediate basin of attraction of an attracting fixed point that lies in $B'$. It follows that the Fatou components $U_m$ are all simply connected.

To conclude the proof of Theorem~\ref{thm:baker-domains-entire}, it only remains to check that the Fatou components $U_m$, $0\leqslant m<N$, are hyperbolic Baker domains. Due to symmetry, it suffices to deal with the case $m=0$. Let $z_0\in U_0$. Since $V_0\subseteq U_0$ is an absorbing region, we can assume without loss of generality that $z_0\in V_0$ and $|z_0|$ is sufficiently large. For $n\in\N$, let 
$$
\epsilon_n:=g(f^{n-1}(z_0))-\hat{g}(f^{n-1}(z_0))
$$
which, by \eqref{eq:bd-entire-approx}, satisfies 
$$
|\epsilon_n|<\varepsilon(|f^{n-1}(z_0)|) \quad \mbox{ as } n\to \infty.
$$
For $n\in\N$, define 
$$
C_n:=\prod_{0<k\leqslant n} \exp \epsilon_k=\exp \sum_{0<k\leqslant n} \epsilon_k,
$$
which represents the quotient $f^n(z_0)/\bigl(z^n\exp(\hat{g}(z_0))\bigr)$. Using the triangle inequality, we obtain
\begin{equation}
|C_n|\leqslant \exp \sum_{0<k\leqslant n} |\epsilon_k|<\exp\sum_{0<k\leqslant n} \varepsilon(|f^{k-1}(z_0)|).
\label{eq:bd-9}
\end{equation}
Next, we are going to show that $|C_n|$ is bounded above for all $n\in\N$. To that end, we find a lower bound for $|f^{k}(z_0)|$ for $k\in\N$ assuming, if necessary, that $|z_0|=r_0$ is sufficiently large. Put $K:=(\sqrt[N]{\lambda}-1)/2>0$. Then $|C_1|>1/K$ for $r_0>0$ sufficiently large and, by \eqref{eq:bd-6} and \eqref{eq:bd-3},
$$
|f(z_0)|=\sqrt[N]{\lambda}|z_0||C_1|\geqslant \frac{\sqrt[N]{\lambda}}{K}r_0=\mu r_0,
$$
with $\mu:=\sqrt[N]{\lambda}/K>1$. Hence, by induction and the symmetry properties of the sets $V_m$, $0\leqslant m<N$,
\begin{equation}
|f^k(z_0)|\geqslant \mu^k r_0\quad \mbox{ for } k\in\N.
\label{eq:bd-7}
\end{equation}
In particular, $z_0\in I(f)$ so, by normality, the periodic Fatou components $U_m$, $0\leqslant m<N$, are Baker domains. We deduce by \eqref{eq:bd-9}, \eqref{eq:bd-4} and \eqref{eq:bd-7} that $|C_n|<e^S$ for all $n\in\N$, where $S<+\infty$ is the sum of the following geometric series
$$
S:=\sum_{k=0}^\infty \frac{1}{(\mu^kr_0)^{N+1}}=\frac{1}{r_0^{N+1}} \sum_{k=0}^\infty \left( \frac{1}{\mu^{N+1}} \right)^k =\frac{\mu^{N+1}}{r_0^{N+1}(\mu^{N+1}-1)}.
$$

Next we use the characterisation of Lemma~\ref{lem:bd-koenig} to show that the Baker domains are hyperbolic. For $n\in\N$, define
$$
c_n=c_n(z_0)=\frac{|f^{(n+1)N}(z_0)-f^{nN}(z_0)|}{\textup{dist}(f^{nN}(z_0),\partial U)}.
$$
We have
$$
f^{nN}(z_0)=C_{nN}\sqrt[N]{\lambda^{nN}z_0^N}=C_{nN}\lambda^{n}z_0 \quad \mbox{ for } n\in\N
$$
and therefore
$$
|f^{(n+1)N}(z_0)-f^{nN}(z_0)|\sim C_\infty\lambda^n(\lambda-1)|z_0|\quad \mbox{ as } n\to\infty,
$$
where $C_\infty:=\lim_{n\to\infty} C_n$. Also, $\textup{dist}(f^{nN}(z_0),\partial U_0)\leqslant e^{S}\lambda^n|z_0|$ and hence if $c:=(\lambda-1)/2>0$, we have $c_n(z_0)>c$ for all $n\in\N$. Thus, by Lemma~\ref{lem:bd-koenig}, the Baker domain $U_0$ is hyperbolic. This completes the proof of Theorem~\ref{thm:baker-domains-entire}.
\end{proof}

Finally we prove Theorem \ref{thm:baker-domains} in which we construct a function~$f$ that is a transcendental self-map of $\C^*$ with $\textup{ind}(f)=n$ that has a cycle of hyperbolic Baker domains in $I_e(f)$, where $e$ is any prescribed periodic essential itinerary $e\in\{0,\infty\}^{\N_0}$.

\begin{proof}[Proof of Theorem \ref{thm:baker-domains}] 
Let $N\in \N$ be the period of $e$ and let $p,q\in\N_0$ denote, respectively, the number of symbols $0$ and $\infty$ in the sequence $e_0e_1\hdots e_{N-1}$, where $p+q=N$; see \eqref{eq:p-and-q}. We modify the proof of Theorem~\ref{thm:baker-domains-entire} to obtain a transcendental self-map of $\C^*$ of the form
$$
f(z):=z^n\exp(g(z)z^{N+1}+h(1/z)/z^{N+1})
$$
that has a hyperbolic Baker domain $U$ in $I_e(f)$, where the entire functions $g, h$~will be constructed using approximation theory. 

We start by defining a collection of $p$ sets $\{V_m^\infty\}_{0\leqslant m<p}$, whose closure in $\CR$ contains infinity. Put $\omega_{p}:=e^{2\pi i/p}$ once again and define
$$
V_{m}^\infty:=\omega_{p}^m\sqrt[p]{W}\subseteq \C\setminus \overline{D(0,\rho)} \quad \mbox{ for } 0\leqslant m<p,
$$
where $W$ is the half-plane from Lemma~\ref{lem:approx-BD} and $\rho:=1+(\sqrt[N]{2}-1)/6$. We denote by $V_\infty$ the union of all $V_m^\infty$, $0\leqslant m<p$.

As before, we define a set $B_\infty$ that will be contained in an immediate basin of attraction of $f$ and put $R_\infty=\R_-$, if $p$ is odd, or $R_\infty=\{z\in\C^*\,:\,\textup{arg}\,z=\pi(1-1/p)\}$, if $p$ is even. Then, let
$$
d_\infty:=\min\{(\sqrt[N]{2}-1)/6,\ \textup{dist}(V_\infty,R_\infty)/4\},
$$
and define the closed connected set
$$
B_\infty:=\{z\in\C\,:\,\textup{dist}\,(z,V_\infty)\geqslant d_\infty \mbox{ and } \textup{dist}\,(z,R_\infty)\geqslant d_\infty\}\setminus D(0,\rho),
$$
which compactly contains the disc $B_\infty':=\overline{D((1+\sqrt[N]{2})/2,(\sqrt[N]{2}-1)/6)}$. Finally, we define the disc $D:=D(0,1/\rho)$, which is contained in $\mathbb D$. We will construct the function $g$ by approximating it on the closed set $F_\infty:=V_\infty\cup B_\infty\cup D$, which satisfies the hypothesis of Lemma \ref{lem:approx-sectors}; namely $\CR\setminus F_\infty$ is connected and $\CR \setminus F_\infty$ is locally connected at infinity, and $F_\infty\subseteq W_\alpha$ with $\alpha=2\pi$ (see Figure~\ref{fig:sketch-bd-cstar-1side}).

\begin{figure}[h!]
\centering
\def\svgwidth{.60\linewidth}
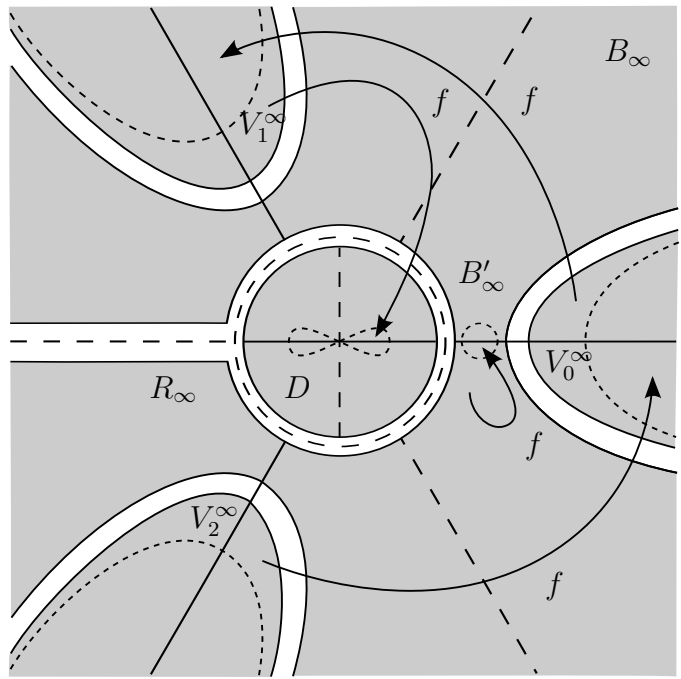
\caption[Sketch of the construction of a transcendental self-map of $\C^*$ that has a cycle of hyperbolic Baker domains I]{Sketch of the construction of the entire function $g$ in the proof of Theorem~\ref{thm:baker-domains} with $e=\overline{\infty\infty00\infty}$. The sets $D$, $B_\infty$ and $V_m^\infty$, \mbox{$0\leqslant m<p$}, are shaded in grey.}
\label{fig:sketch-bd-cstar-1side} 
\end{figure}

Similarly, we define a set $B_0$ and a collection of $q$ unbounded sets $\{V_m^0\}_{0\leqslant m<q}$ by using the same procedure as above, just replacing $p$ by $q$, and then, if $V_0$ is the union of all $V_m^0$, $0\leqslant m<q$, we put $F_0:=V_0\cup B_0\cup D$. The Fatou set of the function $f$ will contain all the sets $V_m^\infty$, $0\leqslant m<p$, and all the sets $\tilde{V}_m^0:=1/V_m^0$, $0\leqslant m<q$, which are unbounded in $\C^*$.

In order to define the functions $\hat{g}\in A(F_\infty)$ and $\hat{h}\in A(F_0)$, we first introduce some notation to describe how $\hat{g}$ and $\hat{h}$ map the components of $V_\infty$ and $V_0$, respectively; we use the same notation as in Theorem~\ref{thm:wandering-domains}. Let $\pi:\{0,\hdots,N-1\}\to \{-q,\hdots,-1,1,\hdots,p\}$ denote the function given by, for $0\leqslant k<N$,
$$
\pi (k):=\left\{
\begin{array}{ll}
\#\{\ell\in\N_0\ :\ e_\ell=\infty \mbox{ for } \ell<k\}+1, & \mbox{ if } e_k=\infty,\vspace{5pt}\\
-\,\#\{\ell\in\N_0\ :\ e_\ell=0 \mbox{ for } \ell<k\}-1, & \mbox{ if } e_k=0.\\
\end{array}
\right.
$$
The function $\pi$ is an ordering of the components of $V_\infty \cup 1/V_0$ according to the sequence $e$. Suppose that $V$ is the starting component; that is, $V=\tilde{V}_0^0$, if $e_0=0$, and $V=V_0^\infty$, if $e_0=\infty$. Then
$$
f^k(V)\subseteq\left\{
\begin{array}{ll}
V_{\pi(k)}^\infty, & \mbox{ if } \pi(k)>0,\vspace{5pt}\\
\tilde{V}_{-\pi(k)}^0, & \mbox{ if } \pi(k)<0.
\end{array}\right.\vspace*{-5pt}
$$
For $m\in \{-q,\hdots,-1,1,\hdots ,p\}$, we define the function
$$
s(m):=\pi(\pi^{-1}(m)+1 \pmod{N}),\vspace*{-5pt}
$$
which describes the image of the component $V_m^\infty$, if $m>0$, and $\tilde{V}_m^0$, if $m<0$, so that the function $f$ to be constructed has a Baker domain that has essential itinerary $e$. More formally, for $0\leqslant m<p$,
$$
f(V_m^\infty)\subseteq\left\{
\begin{array}{ll}
V_{s(m)}^\infty, & \mbox{ if } s(m)>0,\vspace{5pt}\\
\tilde{V}_{-s(m)}^0, & \mbox{ if } s(m)<0;
\end{array}
\right.\vspace*{-5pt}
$$ 
and, for $0\leqslant m<q$, 
$$
f(\tilde{V}_m^0)\subseteq\left\{
\begin{array}{ll}
 V_{s(-m)}^\infty, & \mbox{ if } s(-m)>0,\vspace{5pt}\\
\tilde{V}_{-s(-m)}^0, &\mbox{ if } s(-m)<0.
 \end{array}
 \right.
 $$
We now give the details of the construction of the entire function $g$ from the function $\hat{g}\in A(F_\infty)$. For $z\in V_m^\infty$, $0\leqslant m<p$, we put 
$$
\hat{g}(z):=\left\{\begin{array}{ll}
\left(\log \left(\omega_{p}^{s(m)}\sqrt[p]{\lambda (z/\omega_{p}^m)^{p}}\right)-n\log z\right)/z^{N+1}, & \mbox{ if } s(m)>0,\vspace{5pt}\\
\left(\log \left(\omega_{p}^{s(m)}/\sqrt[p]{\lambda (z/\omega_{p}^m)^{p}}\right)-n\log z\right)/z^{N+1}, & \mbox{ if } s(m)<0,
\end{array}\right.
$$
for $z\in B_\infty$, we put $\hat{g}(z):=(\log(1+(\sqrt[n]{2}-1)/2)-n\log z)/z^{N+1}$ and, for $z\in D$, we put $\hat{g}(z):=0$, where we have taken an analytic branch of the logarithm defined on $\C^*\setminus R_\infty$ and hence on $V_\infty\cup B_\infty$ (see Figure~\ref{fig:sketch-bd-cstar-1side}). Then $\hat{g}\in A(F_\infty)$. For $r>0$, we define the positive continuous function $\varepsilon_\infty$ by
$$
\varepsilon_\infty(r):=\min\{d'_\infty,\ k_\infty^{-(N+1)},\ r^{-(N+1)}\} /(2r^{N+1})
$$
where the constant $d'_\infty>0$ is so small that $|e^z-1|<d_\infty$ for $|z|<d'_\infty$ and the constant $k_\infty>0$ is so large that, for all $z\in \log T_\lambda(W)$ with $\textup{Re}\,z<k_\infty$, the disc $D(z,k_\infty^{-(N+1)})$ is compactly contained in $\log W$ and, moreover, if $\delta_{N}(r)$ is the function from Lemma~\ref{lem:approx-root-BD}, then 
$$
\varepsilon_\infty(r)\cdot 2r^{N+1}<\delta_{N}(\ln (\lambda r)) \quad \mbox{ for } r\geqslant k_\infty, 
$$
which, as before, is possible since 
$$
\delta_{N}(\ln (\lambda r)) \sim \frac{2(\lambda-1)}{N\lambda^Nr^N} \quad \mbox{ as } r\to+\infty.
$$
Since $\varepsilon_\infty$ satisfies 
$$
\int_1^{+\infty} r^{-3/2}\ln\varepsilon_\infty(r)dt>-\infty,
$$
by Lemma \ref{lem:approx-sectors} (with $\alpha=2\pi$), there is an entire function~$g$ such that
\begin{equation}
|g(z)-\hat{g}(z)|<\left\{\begin{array}{ll}
\varepsilon_\infty(|z|) & \mbox{ for } z\in V_\infty\cup B_\infty,\vspace{5pt}\\
1/2 & \mbox{ for } z\in D.
\end{array}\right.
\label{eq:bd-entire-approx}
\end{equation}

Similarly, we can construct an entire function $h$ that approximates a function $\hat{h}\in A(F_0)$ so that the function
$$
\begin{array}{rl}
f(z):=&\hspace{-6pt}z^n\exp(g(z)z^{N+1}+h(1/z)/z^{N+1})\vspace{5pt}\\
=&\hspace{-6pt}z^n\exp(\hat{g}(z)z^{N+1})\exp(\hat{h}(1/z)/z^{N+1}) \cdot \vspace{5pt}\\
&\cdot \exp((g(z)-\hat{g}(z))z^{N+1})\exp((h(z)-\hat{h}(z))/z^{N+1})
\end{array}
$$
has the desired properties. Observe that if $z\in V_\infty\cup B_\infty$, then $1/z\in D$ and if $1/z\in V_0\cup B_0$, then $z\in D$. Thus, $\hat{h}(1/z)=0$ for $z\in V_\infty\cup B_\infty$ and
$$
\begin{array}{c}
|\hat{h}(1/z)/z^{N+1}+(g(z)-\hat{g}(z))z^{N+1}+(h(z)-\hat{h}(z))/z^{N+1}|\leqslant\vspace{5pt}\\
\leqslant 0+1/(2|z|^{N+1})+1/(2|z|^{N+1})=1/|z|^{N+1}
\end{array}
$$
for $z\in V_\infty\cup B_\infty$. 

\begin{figure}[ht!]
\centering
\def\svgwidth{.60\linewidth}
\input{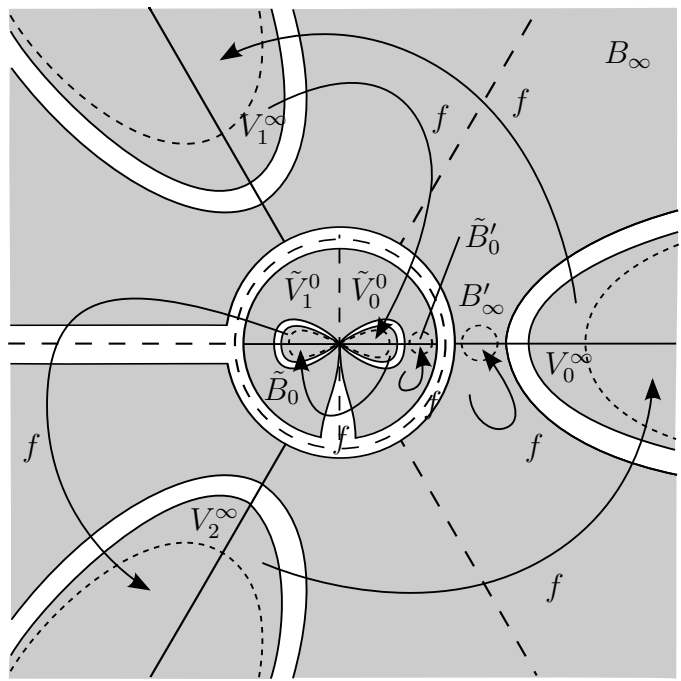}
\caption[Sketch of the construction of a transcendental self-map of $\C^*$ that has a cycle of hyperbolic Baker domains II]{Sketch of the construction of the function $f$ in the proof of Theorem~\ref{thm:baker-domains} with $e=\overline{\infty\infty00\infty}$. 
}
\label{fig:sketch-bd-cstar-2sides} 
\end{figure}

Finally, a similar argument to that in the proof of Theorem~\ref{thm:baker-domains-entire} shows that the Fatou components that we have constructed are hyperbolic Baker domains; we omit the details.
\end{proof}

\red{

}

\bibliography{bibliography}

\newcommand{\noopsort}[1]{}
\providecommand{\bysame}{\leavevmode\hbox to3em{\hrulefill}\thinspace}
\providecommand{\MR}{\relax\ifhmode\unskip\space\fi MR }
\providecommand{\MRhref}[2]{%
  \href{http://www.ams.org/mathscinet-getitem?mr=#1}{#2}
}
\providecommand{\href}[2]{#2}
\begin{thebibliography}{{Mar}16b}

\bibitem[Ara64]{arakeljan64}
N.~U. Arakelyan, \emph{Uniform approximation on closed sets by entire
  functions}, Izv. Akad. Nauk SSSR Ser. Mat. \textbf{28} (1964), 1187--1206.

\bibitem[Bak63]{baker63}
I.~N. Baker, \emph{Multiply connected domains of normality in iteration
  theory}, Math. Z. \textbf{81} (1963), 206--214.

\bibitem[Bak76]{baker76}
\bysame, \emph{An entire function which has wandering domains}, J. Austral.
  Math. Soc. Ser. A \textbf{22} (1976), no.~2, 173--176.

\bibitem[Bak84]{baker84}
\bysame, \emph{Wandering domains in the iteration of entire functions}, Proc.
  London Math. Soc. (3) \textbf{49} (1984), no.~3, 563--576.

\bibitem[Bak87]{baker87}
\bysame, \emph{Wandering domains for maps of the punctured plane}, Ann. Acad.
  Sci. Fenn. Ser. A I Math. \textbf{12} (1987), no.~2, 191--198.

\bibitem[BD98]{baker-dominguez98}
I.~N. Baker and P.~{Dom\'inguez-Soto}, \emph{Analytic self-maps of the
  punctured plane}, Complex Variables Theory Appl. \textbf{37} (1998), no.~1-4,
  67--91.

\bibitem[Ber93]{bergweiler93}
W.~Bergweiler, \emph{Iteration of meromorphic functions}, Bulletin of the
  American Mathematical Society \textbf{29} (1993), no.~2, 151--188.

\bibitem[Ber95]{bergweiler95}
\bysame, \emph{On the {J}ulia set of analytic self-maps of the punctured
  plane}, Analysis \textbf{15} (1995), no.~3, 251--256.

\bibitem[Bis15]{bishop15}
C.~J. Bishop, \emph{Constructing entire functions by quasiconformal folding},
  Acta Math. \textbf{214} (2015), no.~1, 1--60.

\bibitem[BRS13]{bergweiler-rippon-stallard13}
W.~Bergweiler, P.~J. Rippon, and G.~M. Stallard, \emph{Multiply connected
  wandering domains of entire functions}, Proc. Lond. Math. Soc. (3)
  \textbf{107} (2013), no.~6, 1261--1301.

\bibitem[BZ12]{bergweiler-zheng12}
W.~Bergweiler and J.-H. Zheng, \emph{Some examples of {B}aker domains},
  Nonlinearity \textbf{25} (2012), no.~4, 1033--1044.

\bibitem[Cow81]{cowen81}
C.~C. Cowen, \emph{Iteration and the solution of functional equations for
  functions analytic in the unit disk}, Trans. Amer. Math. Soc. \textbf{265}
  (1981), no.~1, 69--95.

\bibitem[EL92]{eremenko-lyubich92}
A.~E. Eremenko and M.~Yu. Lyubich, \emph{Dynamical properties of some classes
  of entire functions}, Ann. Inst. Fourier (Grenoble) \textbf{42} (1992),
  no.~4, 989--1020.

\bibitem[Ere89]{eremenko89}
A.~E. Eremenko, \emph{On the iteration of entire functions}, Dynamical systems
  and ergodic theory ({W}arsaw, 1986), Banach Center Publ., vol.~23, PWN,
  Warsaw, 1989, pp.~339--345.

\bibitem[Fag99]{fagella99}
N.~Fagella, \emph{Dynamics of the complex standard family}, J. Math. Anal.
  Appl. \textbf{229} (1999), no.~1, 1--31.

\bibitem[Fat26]{fatou26}
P.~Fatou, \emph{Sur l'it\'eration des fonctions transcendantes enti\`eres},
  Acta Math. \textbf{47} (1926), no.~4, 337--370.

\bibitem[FH06]{fagella-henriksen06}
N.~Fagella and C.~Henriksen, \emph{Deformation of entire functions with {B}aker
  domains}, Discrete Contin. Dyn. Syst. \textbf{15} (2006), no.~2, 379--394.

\bibitem[FM17]{fagella-martipete}
N.~Fagella and D.~{Mart\'i-Pete}, \emph{Dynamic rays of bounded-type
  transcendental self-maps of the punctured plane}, Discrete Contin. Dyn. Syst.
  Ser. A \textbf{37} (2017), 3123--3160.

\bibitem[Gai87]{gaier87}
D.~Gaier, \emph{Lectures on complex approximation}, Birkh\"auser Boston, Inc.,
  Boston, MA, 1987, Translated from the German by Renate McLaughlin.

\bibitem[Gau13]{gauthier13}
P.~M. Gauthier, \emph{Approximating the {R}iemann zeta-function by strongly
  recurrent functions}, Blaschke products and their applications, Fields Inst.
  Commun., vol.~65, Springer, New York, 2013, pp.~31--42.

\bibitem[KL03]{keen-lakic03}
L.~Keen and N.~Lakic, \emph{Forward iterated function systems}, Complex
  dynamics and related topics: lectures from the {M}orningside {C}enter of
  {M}athematics, New Stud. Adv. Math., vol.~5, Int. Press, Somerville, MA,
  2003, pp.~292--299.

\bibitem[K{\"o}n99]{koenig99}
H.~K{\"o}nig, \emph{Conformal conjugacies in {B}aker domains}, J. London Math.
  Soc. (2) \textbf{59} (1999), no.~1, 153--170.

\bibitem[Kot90]{kotus90}
J.~Kotus, \emph{The domains of normality of holomorphic self-maps of {${\bf
  C}^*$}}, Ann. Acad. Sci. Fenn. Ser. A I Math. \textbf{15} (1990), no.~2,
  329--340.

\bibitem[{Mar}]{martipete4}
D.~{Mart\'i-Pete}, \emph{{\noopsort{U}}{E}scaping points and semiconjugation of
  holomorphic self-maps of the punctured plane}, in preparation.

\bibitem[{Mar}16a]{martipete1}
\bysame, \emph{The escaping set of transcendental self-maps of the punctured
  plane}, to appear in Ergod. Th. \& Dynam. Sys., 1--22, DOI
  10.1017/etds.2016.36. Published online on October 13, 2016.

\bibitem[{Mar}16b]{martipete}
\bysame, \emph{Structural theorems for holomorphic self-maps of the punctured
  plane}, Ph.D. thesis, The Open University, 2016.

\bibitem[Mil06]{milnor06}
J.~W. Milnor, \emph{Dynamics in one complex variable}, third ed., Annals of
  {M}athematics Studies, vol. 160, Princeton University Press, Princeton, 2006.

\bibitem[Muk91]{mukhamedshin91}
A.~N. Mukhamedshin, \emph{Mapping of a punctured plane with wandering domains},
  Sibirsk. Mat. Zh. \textbf{32} (1991), no.~2, 184--187, 214.

\bibitem[Ner71]{nersesjan71}
A.~A. Nersesyan, \emph{Carleman sets}, Izv. Akad. Nauk Armjan. SSR Ser. Mat.
  \textbf{6} (1971), no.~6, 465--471.

\bibitem[R{\r{a}}d53]{radstrom53}
H.~R{\r{a}}dstr\"om, \emph{On the iteration of analytic functions.}, Math.
  Scand. \textbf{1} (1953), 85--92.

\bibitem[Rip08]{rippon08}
P.~J. Rippon, \emph{Baker domains}, Transcendental {D}ynamics and {C}omplex
  {A}nalysis, London Math. Soc. Lecture Note Ser., vol. 348, Cambridge Univ.
  Press, Cambridge, 2008, pp.~371--395.

\bibitem[RS08]{rippon-stallard08}
P.~J. Rippon and G.~M. Stallard, \emph{On multiply connected wandering domains
  of meromorphic functions}, J. Lond. Math. Soc. (2) \textbf{77} (2008), no.~2,
  405--423.

\end{thebibliography}

\end{document}